\documentclass[10pt,reqno]{amsart}

\usepackage{amsmath,amsfonts,amssymb,amscd,amsthm,amsbsy,calc,enumerate,color,graphicx,verbatim,cite}
\usepackage{color}

\usepackage{float}

\usepackage[colorlinks,citecolor=blue,linkcolor=blue]{hyperref}

\usepackage[shortalphabetic,initials]{amsrefs}

\renewcommand{\MultipleCiteKeyWarning}[2]{}

\usepackage[margin=1in]{geometry}

\theoremstyle{plain}
\newtheorem{theorem}{Theorem}[section]
\newtheorem{definition}[theorem]{Definition}
\newtheorem{lemma}[theorem]{Lemma}
\newtheorem{claim}[theorem]{Claim}
\newtheorem{corollary}[theorem]{Corollary}
\newtheorem{proposition}[theorem]{Proposition}
\theoremstyle{definition}

\theoremstyle{remark}
\newtheorem{remark}[theorem]{Remark}

\newcommand{\R}{{\mathbb R}}
\newcommand{\e}{\varepsilon}

\newcommand{\abs}[1]{\left| #1 \right|}
\newcommand{\pth}[1]{\left( #1 \right)}

\newcommand{\set}[1]{\left\{ #1 \right\}}

\newcommand{\norm}[1]{\lVert#1\rVert}

\definecolor{darkgreen}{rgb}{0,0.4,0}

\newcommand{\cl}[1]{\overline{#1}}
\newcommand{\Rn}{{\mathbb{R}^n}}
\newcommand{\supp}{\operatorname{supp}}
\newcommand{\interior}{\operatorname{int}}
\newcommand{\divo}{\operatorname{div}}
\newcommand{\strictordwrt}[1]{\prec_{#1}}

\newcommand{\halflimsup}{\operatorname*{limsup^*}}
\newcommand{\halfliminf}{\operatorname*{liminf_*}}
\newcommand{\alphlist}{\renewcommand{\labelenumi}{(\alph{enumi})}}

\newcommand{\Per}{\operatorname{Per}}
\newcommand{\compl}{\mathsf{c}}

\numberwithin{equation}{section}
\newcommand{\vecb}{\vec{b}}

\newcommand{\ubar}[1]{\text{\b{$#1$}}}

\title{Singular limit of the porous medium equation with a drift}

\author[I. Kim]{Inwon Kim}
\address[I. Kim]{Department of Mathematics, UCLA, LA CA 90095 USA.}
\email{ikim@math.ucla.edu}

\author[N. Po\v{z}\'{a}r]{Norbert Po\v{z}\'{a}r}
\address[N. Po\v{z}\'{a}r]{Faculty of Mathematics and Physics, Institute of Science and Engineering, Kanazawa University,
Kakuma town, Kanazawa, Ishikawa 920-1192, Japan.}
\email{npozar@se.kanazawa-u.ac.jp}

\author[B. Woodhouse]{Brent Woodhouse}
\address[B. Woodhouse]{Department of Mathematics, UCLA, LA CA 90095 USA.}
\email{bwoodhouse729@math.ucla.edu}

\date{\today}

\begin{document}

\begin{abstract}
 We study the ``stiff pressure limit'' of a nonlinear drift-diffusion equation, where the density
 is constrained to stay below the maximal value one. The challenge lies in the presence of a drift and the consequent lack of monotonicity in time.  In the limit
 a Hele-Shaw-type free boundary problem emerges, which describes the evolution of the congested zone where density equals one.  We discuss pointwise convergence of the densities as well as the $BV$ regularity of the limiting free boundary.
\end{abstract}

\maketitle

\section{Introduction}

Let $\rho_m(x,t)$ solve the drift-diffusion problem
\begin{equation}\label{pme}
\rho_t - \Delta (\rho^m) + \divo (\rho\vec{b})=f\rho \quad \hbox{ in } Q:=\R^n\times (0,\infty),
\end{equation}
with initial data $\rho^0_m \in L^1(\R^n)$ and exponent $m > 1$. The
nonlinear diffusion term in \eqref{pme} represents an anti-congestion effect, and
has been used in many physical applications including fluids, biological aggregation and population dynamics
(\cite{BGHP,BH,Mur,HW,W,TBL}) and more recently in the context of tumor growth (\cite{PQV}).
\medskip

It is instructive to write \eqref{pme} in the form of a continuity equation,
\begin{equation}\label{pressure00}
\rho_t - \divo(\rho(\nabla p_m -\vec{b}))= f\rho,
\end{equation}
where $p_m$ denotes the pressure variable, $p_m = P_m(\rho_m) := \frac{m}{m-1} (\rho_m)^{m-1}$.  $p_m$ satisfies the \emph{pressure-form equation}
\begin{align}
  \label{pmepressure}
  p_t - (m - 1)p(\Delta p + f - \divo \vec b)- \nabla p\cdot(\nabla p - \vec b) = 0.
\end{align}
In the above equations, the operators $\Delta, \divo, \nabla$ are taken in the space variable.

\medskip

We are interested in identifying the behavior of $\rho_m$ in the ``stiff pressure" limit $m \to
\infty$. The motivation for studying this limit comes from various physical applications, we refer to \cite{CF87} and to more recent articles \cite{MRS}, \cite{PQV}. The limit problem can be interpreted as imposing a maximum value
constraint $\rho\leq 1$ on the density $\rho$ while it is transported by the vector field $\vec{b}$ and created
by the source $f$. As we will discuss below,  the limit of the pressure variable $p_m$ plays the
role of a Lagrange multiplier for the constraint, and it is supported in the {\it congested
zone } where the limiting density achieves the maximum value $1$.  Thus in the limit $m\to\infty$
we are led to a free boundary problem that describes the evolution of the congested zone in terms of the pressure.

\medskip

To introduce the limiting free boundary problem, some assumptions are in order. First we assume that the
drift and source terms are sufficiently regular, which allows pointwise description of the free
boundary movement. We assume that $\vec{b}(x,t) :Q\to \R^n$ is a $C^2$ vector field, and $f: Q \to
\R$ is
continuous. In addition we assume that
\begin{align}
  \label{F}
  F := f - \divo \vec{b} >0.
\end{align}
This last assumption yields certain monotonicity properties of the limit density along the streamlines of $\vec{b}$;  we will discuss more on \eqref{F} below.

\medskip

Before summarizing the main results, let us introduce the class of initial data for the limiting
problem that we consider (Figure~\ref{fig:nucleation}). Denoting by $\chi_A$ the characteristic function of the set $A$, we say  $\rho^0\in L^1(\R^n)$ is \emph{regular} if it is of the form
\begin{equation}\label{initial}
  \rho^0 = \max(\chi_{\Omega^0}, \rho^{E,0}),
\end{equation}
where $\Omega^0 \subset \Rn$ is a compact set such that $\Omega^0 =
\overline{\interior \Omega^0}$ and $\rho^{E, 0} \in C_c(\Rn)$ with $0 \leq \rho^{E, 0} <1$.
\medskip

Note that our initial data includes any continuous initial data between zero and one with compact support, as well as any characteristic function of a regular open, bounded set.

\medskip

 In view of \eqref{pmepressure} we can
perform a formal calculation, similar to the one in Section~1 of  \cite{KP},  to conclude
that the limiting pressure $p \geq 0$ solves a quasi-static, Hele-Shaw-type problem
\begin{align}
  \label{hs}
\left\{\begin{aligned}
-\Delta p &= F &&\text{ in } \{p>0\}, \\
V &= \pth{-\dfrac{\nabla p}{(1-\rho^E)_+} + \vec{b}} \cdot \nu  &&\text{ on } \partial\{p>0\}.\\
\end{aligned}\right.
\end{align}
Here $V=V_{x,t}$ is the normal velocity of the set $\{p>0\}$ at $(x,t)\in \partial\{p>0\}$,
$\nu=\nu(x,t)$ denotes the unit outer spatial normal at the \emph{free boundary} $\partial\{p>0\}$. The \emph{external density}
$\rho^E= \rho^E(x,t)$ corresponds to the expected limit density outside of the congested zone $\{p>0\}$,  and solves the transport equation
\begin{align}
  \label{te}
\rho_t + \divo(\rho\vec{b})=f\rho \quad \hbox{ in } Q, \qquad \rho(\cdot,0) = \rho^{E,0},
\end{align}
where the initial data $\rho^{E,0}$ is given in \eqref{initial}.  The notation $\frac1{(s)_+}$
denotes $\frac1s$ when $s>0$ and $+\infty$ otherwise.

\medskip

Now we are ready to state our main theorem.

\begin{theorem}\label{main}

Let $\rho^0$ be regular in the sense of \eqref{initial}. Let $\rho_m$ and $p_m = P_m(\rho_m)$ be the solutions of \eqref{pme} with initial data $\rho_m^0$ such that $\rho_m^0\to \rho^0$ in $L^1(\R^n)$. Then the following holds:

\begin{itemize}
\item[(a)]  (Corollary~\ref{density:gen}) $\rho_m$ converges in $L^1_{loc}(\R^n\times [0,\infty))$ to  $\rho$ given by
\begin{align}
  \label{limit-rho}
\rho:= \chi_{\Omega} + \rho^E\chi_{\Omega^\compl},
\end{align}
where $\rho^E$ is the solution of \eqref{te} with initial data $\rho^{E,0}$, $\Omega:= \overline{\{p>0\}}$ and $p$ is a viscosity solution of \eqref{hs}  with initial density $\rho^0$, defined in
Definition~\ref{def:visc-sol}. The congested zone $\Omega$ is unique while $p$ may not be (see Corollary~\ref{cor:unique}). \\

\item[(b)] (Corollary~\ref{density}) Suppose in addition that $\rho_m^0$ converges to $\rho^0$ in terms
  of semi-continuous envelopes, i.e., they satisfy \eqref{initial:convergence}. Then
  $\rho_m$ converges to $\rho$  locally uniformly in $(\R^n\times[0,\infty)) \setminus \partial{\{p>0\}}$. \\

\item[(c)] (Corollary~\ref{pressure}) For $p$ and $\rho_m^0$ as given in $(a)$--$(b)$, the pressure variable $p_m$ uniformly converges to $p$ in any local neighborhood $\mathcal{N}$ of $\Omega$ where $\Omega$ has sufficiently regular (e.g. Lipschitz) boundaries.\\

\item[(d)] (Proposition~\ref{le:perimeter}) If $\rho^E$ is strictly below $1$ outside of $\{p(\cdot,t)>0\}$ in a local neighborhood $\mathcal{N}$, then the perimeter of the congested zone in $\mathcal{N}$, \,$Per(\{p(\cdot,t)>0\}, \mathcal{N})$, is finite.
\end{itemize}

\end{theorem}

\medskip

Note that $\rho^E$ may reach the value $1$ away from the existing congested zone from previous times, nucleating new regions of congestion, see Figure~\ref{fig:nucleation}. When $\rho^E$ reaches $1$ at the free boundary, the velocity law in \eqref{hs} indicates that the boundary will
move with an infinite speed or even discontinuously in time.
\begin{figure}
  \centering
  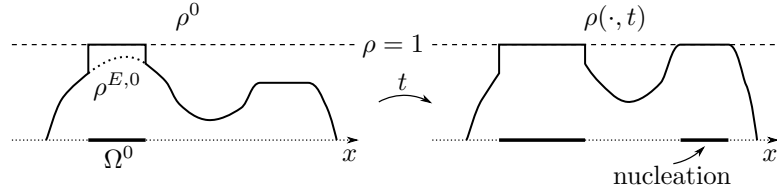
  \caption{Left: Regular initial data. Right: Nucleation of the congested zone (thick line).}
  \label{fig:nucleation}
\end{figure}
Thus for \eqref{hs}  a nucleation of the
pressure zone as well as an infinite speed of propagation are generic phenomena.  This leads to
interesting singularities and the necessity of weak solutions (in our case, viscosity solutions)
that allow both the description of the free boundary evolution as well as the discontinuity. Let us
also mention that with certain drift fields, for instance the gradient of a scalar function in the
neighborhood of its saddle point, the pressure zone may experience neck-pinching, adding to the diversity of topological singularities.

\medskip

\subsection*{Literature} Our discussion here focuses on articles addressing the limit $m\to\infty$. For the case $\vec{b}=0$, there is a vast literature for different weak solutions and regularity theory of \eqref{pme} : we refer to the book \cite{Vazquez} and the references therein. Until recently, the limit $m\to\infty$ has been studied only when both $\vec{b}$
and $f$ are zero. In this case, the problem \eqref{hs} reduces to the classical Hele-Shaw problem. The limit was first considered in
\cite{CF87,EHKO}, see also \cite{BC}, on $\Rn$. See \cite{GQ01,GQ03,K03} for results
on a subset of $\Rn$ with fixed boundary when $\rho^0$ is a {\it patch}, i.e., when it is a characteristic function of a
compact set. In this case there is zero external density $\rho^E$ in \eqref{hs}, i.e., there is
only congested zones evolving in time. This yields the finite speed of propagation property of the congested zone which makes \eqref{hs} more stable and easier to analyze.
\medskip

 The limit problem with positive source ($f>0$ and $\vec{b}=0$) has been studied first in \cite{PQV} in the
context of a mechanical tumor growth model. In this setting, the characterization of
the problem with \eqref{hs}, in the presence of the external density, is shown in \cite{KP} and in \cite{MPQ}
around the same time. All of these articles strongly use the fact that $(\rho_m)_t\geq 0$, which is
a consequence of the zero drift and a positive source term in \eqref{pme}. In particular this leads to an
Aronson--B\'{e}nilan-type semi-convexity estimate on $p_m$, which yields compactness for the pressure variables and their supports as $m\to\infty$. The lack of such estimate appears to be the main challenge in the study of regularity properties of the interface $\{\rho_m>0\}$ when there is a drift.  The monotonicity of solutions is also essential in the
viscosity solutions approach employed in \cite{KP}; we will discuss the approach taken in \cite{KP}
in more detail below.

\medskip

Let us also mention that, when $f=0$ and $\vec{b}$ is a potential velocity field, \eqref{pme} can be posed in the setting of
 gradient flows in the Wasserstein space of probability measures (\cite{O}).
 The limiting problem in this case can be also posed as the gradient flow solution of the transport
 equation \eqref{te} with $L^\infty$ constraint on the density, $\rho \leq 1$ (see \cite{MRS} where
 the problem is introduced in the setting of crowd motion). See \cite{AKY,CKY} where the limit density is characterized with \eqref{hs} in the case of compressive potential and patch solutions.

\medskip


\subsection*{Main challenges and ingredients of the proof}



As in \cite{KP}, our main strategy is to use viscosity solutions theory to both show the existence
of the limit as $m\to\infty$ and to verify that the limit problem is a solution of
\eqref{hs}--\eqref{te} by considering the half-limits of the pressure and density variables defined in \eqref{half_limits}.  Due to the lack of uniform semi-convexity property of $\rho_m$, the finer steps we need to take however are different and much more complex
compared to the standard procedure. Below are the two main ingredients in the convergence proof
that are new in this paper.




\medskip


 First, we will perturb the radial solutions given in \cite{KP} to
construct the barriers to act as test functions for the limit densities.
The construction of such barriers is one of the challenges in passing in the limit $m \to \infty$.
In particular, it is difficult to capture the behavior of solutions as $m \to \infty$ near $\rho_m
\sim 1$, which corresponds to the scenario in \eqref{hs} where $\rho^E$ reaches $1$ from below.
 The ability to perturb the barriers of the simpler problem in \cite{KP} is a crucial advantage of the viscosity
solution method.

Second and more importantly,  our assumption \eqref{F} is used to conclude that
 a streamline cannot leave the congested region for the limit density $\{\rho = 1\}$, it
can only possibly enter it from the ``exterior'' region $\{\rho<1\}$. In particular this way the
pressure does not have an effect on the evolution of the density along the streamline in the
exterior region, and therefore the transport equation \eqref{te} determines $\rho=\rho^E$ outside of $\{\rho=1\}$.
This is related to the monotonicity of the limit problem along the streamlines and
replaces the in-time monotonicity that was important for the zero drift case. Unfortunately we are
not able to fully obtain this property until the full convergence result has been established, and thus we
need to start with a weaker version of the property to proceed in Section~\ref{sec:almost-cp}.

\medskip

 In terms of the structure, the most notable difference from the standard viscosity solutions approach lies in that the viscosity solution property for the pressure
half-limits $\bar{p}$ and $\ubar{p}$, Lemma~\ref{viscosity_limit}, are fully proved only after showing a
comparison principle for the limits (Theorem~\ref{th:almost_cp}) and obtaining the $L^1$ convergence for the density variable (Lemma~\ref{approximation}). Indeed
the general viscosity solution theory is only used in our analysis for the characterization of the limit
density in terms of \eqref{hs}--\eqref{te}. The comparison principle as well as the convergence argument, as
mentioned above, are built upon certain monotonicity properties of the density half-limits
$\bar\rho$, $\ubar\rho$ along characteristic paths (Lemma~\ref{le:char}).

\medskip

\subsection*{Open questions}

\medskip

$\circ$ {\it Removing \eqref{F}:}\,\,  When $F$ changes sign, we no longer expect $\rho^E$ to solve \eqref{te} entirely in terms of the initial data. Hence a new description of the limit problem, as well as new ideas, is necessary to investigate the limit $m\to\infty$ in terms of the evolution of the congested zone.

\medskip

$\circ$ {\it General initial density:} \,\, Here we assume that the initial density $\rho^0$ is regular in the sense of \eqref{initial}. Relaxing
this assumption is plausible but some generalizations are beyond the scope of the framework given
in the paper. For instance when the region $\{\rho^0=1\}$ is not compactly supported, there is an
additional issue of dealing with the growth of the pressure variable at infinity as the solution
evolves. This is likely to be a technical difficulty but we do not investigate it. A more interesting
question arises with the initial data that is larger than $1$ at some points. In such cases there
is a jump in the solution at $t = 0$ in the limit $m \to \infty$ which adds another challenge in
the analysis. In fact the result in \cite{CF87} indicates that the portion of the initial density
over 1 gets spread out immediately to transform into the ``nearest density" under the constraint
$\rho \leq 1$. We do not pursue this interesting aspect of the problem in this paper.

\medskip

$\circ$ {\it Free boundary regularity:} \,\,  Regularity properties of the interfaces $\partial\{p_m>0\}$ and $\partial\{p>0\}$ stays open except for the case $\vec{b}, f=0$ (\cite{CJK,KKV,Vazquez}) and for particular cases of traveling wave solutions with a shear flow (\cite{MNR}). We also mention a numerical result in \cite{Mon} which shows singularity formulation on $\partial\{p_m>0\}$ with a smooth choice of vector field $\vec{b}$.

\medskip

$\circ$ {\it BV regularity for the limit density:} \,\, In general we expect the limit density to be $BV$-regular as indicated by the gradient-flow based analysis of \cite{DPMSV}, but this remains open in our setting. Our result in section 5 only establishes this in the case of external density $\rho^E$ strictly below $1$.

\medskip

\subsection*{Outline}
Before we begin the analysis, let us  give a brief outline of the paper.
Section~\ref{sec:preliminaries} contains notions and preliminary results to be used in the rest of the paper, including the $L^1$ contraction properties for weak solutions of \eqref{pme}.

\medskip

In Section~\ref{sec:almost-cp} we proceed by first showing a comparison result (``almost comparison'') between two
pressure half-limits with strictly ordered initial data (Theorem~\ref{th:almost_cp}). Many properties of the density half-limits are derived in this section, including some monotonicity properties.

\medskip

Section~\ref{sec:convergence} builds on the comparison result to obtain main convergence results.
In Section~\ref{sec:density-convergence} using
the $L^1$~contraction argument we deduce that the density $\rho_m(\cdot,t)$ converges to the limit
density $\rho(\cdot,t)$ in $L^1(\R^n)$ for all $t>0$, as $m\to\infty$. In
Section~\ref{sec:characterization}  we establish
that the congested zones and the pressure supports from each half-limits all coincide, \eqref{support}, which is essential in showing that the congested zone evolves by the free boundary problem \eqref{hs} where $p$ can be identified with the pressure limit. It follows that $\rho_m$ locally uniformly converges to $\rho$ away from the boundary of the congested zone $\{\rho=1\}$. Lastly we invoke Perron's method to show that the congested zone can be characterized as the unique support for any viscosity solution of \eqref{hs}, paired with the corresponding initial data (Corollary~\ref{cor:unique}).

\medskip

In Section~\ref{sec:bv} we focus on the regularity of the set $\{\rho=1\}$ in a local neighborhood where the
external density $\rho^E$ given by the transport equation \eqref{te} stays strictly below $1$. In
such settings we show that $\{\rho=1\}$ has finite perimeter, and thus it follows that $\rho_m$
locally uniformly converges to $\rho$ except on a set of lower dimension. Our assumption on
$\rho^E$ leaves out the more singular scenario when $\rho^E$ is allowed to nucleate an additional
congested zone by increasing to $1$. The last part of the section discusses two examples where
$\rho^E\equiv 0$ after finite time, outside of the congested zone.

\medskip

The appendix deals with the construction of test functions, including the perturbed radial test
functions necessary for understanding the behavior of solutions at density $1$.

\medskip

\section{Preliminaries}
\label{sec:preliminaries}

We will often take a closure of a space-time set and then its time-slice. We use the notation
\begin{align*}
  \cl A_t := \set{x: (x,t) \in \cl A}, \qquad t \in \R,\ A \subset \Rn\times\R.
\end{align*}

We need to discuss here which solutions of \eqref{pme} we consider and what is known about them, as well as initial data.

\subsection{Half-relaxed limits and the initial data}

Let us review the notation first. The \emph{half-relaxed limits} or \emph{half-limits} $\halfliminf$ and $\halflimsup$
of a sequence of locally bounded functions $u_m = u_m(x, t)$ are defined as
\begin{align}\label{half_limits}
  \halflimsup_{m \to \infty} u_m(x, t) := \limsup_{\substack{m\to\infty\\(y,s)\to(x,t)}} u_m(y,s),
  \qquad
  \halfliminf_{m \to \infty} u_m(x, t) := \liminf_{\substack{m\to\infty\\(y,s)\to(x,t)}} u_m(y,s).
\end{align}
It is well-known that $\halflimsup u_m$ is upper semi-continuous (USC) and $\halfliminf u_m$ is lower
semi-continuous (LSC).

Throughout the paper, we will assume that the initial data $\rho^0_m$ for \eqref{pme} converge in
the sense of the half-relaxed limits, that is,
  \begin{equation}\label{initial:convergence}
(\rho^0)_* \leq \halfliminf_{m\to\infty} \rho_m^0, \qquad \halflimsup_{m \to\infty}\rho_m^0 \leq
(\rho^0)^*,
\end{equation}
Here $u^*$ and $u_*$ respectively
denote the upper and lower semi-continuous envelopes: our condition \eqref{initial:convergence} is
a generalization of the uniform convergence for discontinuous initial data $\rho^0$.

\subsection{Notion of solutions for the porous medium equation}

We use the notion of weak solutions of \eqref{pme}, similar to \cite[Section~5.2]{Vazquez}:

Let $Q := \mathbb{R}^n \times (0, \infty)$, and $\rho_0$ take the form \eqref{initial}.  We say that $\rho_m = \rho$ defined on $Q$ is a weak solution of \eqref{pme} if
\begin{itemize}
\item[(i)] $\rho \in L^1(Q)$ and $\rho^m \in L^1(0, \infty; W_0^{1, 1}(\mathbb{R}^n))$,

\item[(ii)] $\rho$ satisfies the identity
$$\iint\limits_{Q} \left\{- \rho \eta_t + \left(\nabla \left(\rho^m\right) - \rho \vec{b}\right) \cdot \nabla \eta \right\} \, dx \, dt = \int_{\mathbb{R}^n} \rho_0(x) \eta(x, 0) \, dx + \iint\limits_{Q} f \rho \eta \, dx \, dt$$
for any function $\eta \in C^1(\cl Q)$ with compact support.

\end{itemize}

Existence can be shown by following \cite[Section~5.4]{Vazquez}.  Uniqueness follows from the $L^1$ contraction property below.  Note that the solution is classical whenever it is positive, due to the regularity of $\vec{b}$ and $f$.

The following lemma can be checked with a parallel proof to \cite[Section~3.2.3]{Vazquez}, we write the proof here for completeness.

\begin{lemma}[$L^1$ contraction]
\label{le:contraction}
  Let $\rho_m$ and $\tilde \rho_m$ be two nonnegative solutions of \eqref{pme} with given initial
  data and source terms $f$, $\tilde f$. Then
\begin{equation}
  \label{contraction}
  \begin{aligned}
  \|\rho_m(\cdot,t) - \tilde{\rho}_m(\cdot,t) \|_{L^1 (\R^n)} &\leq e^{t
  \max(\norm{f}_\infty,\norm{\tilde f}_\infty)}\bigg(\|\rho_m(\cdot,0) -
    \tilde{\rho}_m(\cdot,0)\|_{L^1(\R^n)} \\
    &\quad+ \norm{(\tilde f - f)_+}_\infty \frac{e^{t
  \norm{f}_\infty}}{\norm{f}_\infty}\norm{\rho_m(\cdot,0)}_{L^1(\Rn)} + \norm{(f - \tilde
f)_+}_\infty \frac{e^{t \norm{\tilde f}_\infty}}{\norm{\tilde f}_\infty}\norm{\tilde
\rho_m(\cdot,0)}_{L^1(\Rn)}\bigg)
  \end{aligned}
\end{equation}
for all $t > 0$, where $\frac{e^{t\lambda}}{\lambda} := t$ for $\lambda = 0$.
\end{lemma}

\begin{proof}
  Following \cite{Vazquez}, by approximation it is enough to prove this inequality for smooth positive solutions on
  a bounded domain with zero boundary data. We
  drop the subscript $m$ in the following and write $\rho(t) = \rho(\cdot, t)$, etc. If $f \leq
  \tilde f$ and $\rho(0)
  \leq \tilde\rho(0)$, then $\rho(t) \leq \tilde \rho(t)$ and by the divergence theorem
  \begin{align*}
    \frac{d}{dt} \int (\tilde \rho(t) - \rho(t)) \;dx &= \int \tilde \rho(t) \tilde f - \rho(t) f
    \;dx = \int (\tilde \rho(t) - \rho(t)) \tilde f + \rho(t) (\tilde f - f) \;dx\\
    &\leq \norm{\tilde f}_\infty \int (\tilde \rho(t) - \rho(t)) \;dx +
    \norm{\tilde f - f}_\infty \int \rho(t) \;dx.
  \end{align*}
  Taking $\tilde\rho \equiv 0$ and $\tilde f \equiv 0$ and obtaining a similar estimate as above,
  Gronwall's inequality yields $\int \rho(t) \;dx \leq e^{t\norm{f}_\infty} \int \rho(0)
  \;dx$.
  Therefore Gronwall's inequality implies
  \begin{align*}
  \int (\tilde \rho(t) - \rho(t)) \;dx \leq e^{t
  \norm{\tilde f}_\infty} \pth{\int (\tilde \rho(0) - \rho(0)) \;dx + \frac{\norm{\tilde f -
  f}_\infty}{\norm{f}_\infty} e^{t \norm{f}_\infty} \int \rho(0) \;dx}.
  \end{align*}
  In general, we let $U$ be the solution of \eqref{pme} with initial data $\max(\rho(0), \tilde
  \rho(0)) \geq \rho(0)$ and source $\max(f, \tilde f)$. We have $U - \rho \geq \max(\rho, \tilde
  \rho) - \rho = (\tilde \rho - \rho)_+$ with equality at $t = 0$. Therefore
  \begin{align*}
    \int (\tilde \rho(t) - \rho(t))_+ \;dx &\leq \int (U(t) - \rho(t)) \;dx \\
&\leq
e^{t\max(\norm{f}_\infty, \norm{\tilde f}_\infty)} \pth{\int (\tilde \rho(0) - \rho(0))_+ \;dx
  + \frac{\norm{(\tilde f -
f)_+}_\infty}{\norm{f}_\infty} e^{t \norm{f}_\infty} \int \rho(0) \;dx,}
  \end{align*}
from which we can deduce \eqref{contraction}.
\end{proof}

Note that the proof of the $L^1$ contraction also yields the comparison principle property.

\begin{lemma}[Comparison principle]
\label{le:comparison-pme}
  Let $\rho_m$ and $\tilde \rho_m$ be two nonnegative solutions of \eqref{pme} with given initial
  data and source terms $f$, $\tilde f$. If $\rho_m(\cdot, 0) \leq \tilde \rho_m(\cdot, 0)$ a.e.
  and $f \leq \tilde f$ a.e. then
  \begin{align*}
    \rho_m \leq \tilde\rho_m \qquad a.e.
  \end{align*}
\end{lemma}

\section{Convergence of PME to HS and an almost comparison}
\label{sec:almost-cp}

The goal of this section is to analyze the half-relaxed limits of $\rho_m$ and $p_m$ using
viscosity solution techniques (arguments using the comparison principle). The main result is the
\emph{almost comparison}, Theorem~\ref{th:almost_cp}, that guarantees the ordering of the
limits of solutions with strictly ordered initial data.
Let us stress that we do not use the definition of viscosity solutions for the limit problem
\eqref{hs}, we only use the comparison principle for the solutions $\rho_m$, $p_m$ of \eqref{pme}.
The notion of viscosity solutions and the comparison principle of \eqref{hs} are only introduced
after the ordering of the limits have been understood.

We first introduce the necessary notation in Section~\ref{sec:convergence-half-relaxed}, including
the definition of half-limits of $\rho_m$ and $p_m$ and the statement of the almost comparison.
The vector field $\vecb$ transports mass along trajectories---streamlines---that form a flow map, whose
properties are discussed in Section~\ref{sec:flow-map}. The limits of $\rho_m$ and $p_m$ are
monotone
along these streamlines, and we use this fact to derive important properties of
the limits in Section~\ref{sec:prop-half-limits}. The proof of almost comparison is given in
Section~\ref{sec:proof-of-support-order}, relying on a perturbation argument and a careful
understanding of the behavior of the congested region in the limit. The comparison principle motivates the definition of
viscosity solutions of \eqref{hs} in Section~\ref{sec:viscosity-solutions}.

\subsection{Almost comparison}
\label{sec:convergence-half-relaxed}

We first consider a weaker result of convergence of half-relaxed limits of solutions of
\eqref{pme}.

We recall that the pressure $p_m = P_m(\rho_m) := \frac{m}{m-1} \rho_m^{m-1}$ satisfies the
\emph{pressure equation} \eqref{pmepressure}.

\medskip

Recall the notion of \emph{regular} initial data in the sense of
\eqref{initial}.
We say that regular initial data
\[\rho^{-,0} = \max(\chi_{\Omega^{-,0}}, \rho^{-,0}_E) \quad \text{and}\quad \rho^{+,0} =
\max(\chi_{\Omega^{+,0}}, \rho^{+,0}_E)\]
are \emph{strictly ordered} if
\begin{align}
  \label{initial-data-order}
  \Omega^{-,0} \subset \interior \Omega^{+,0} \qquad \text{and} \qquad \rho^{-,0}_E < \rho^{+,0}_E \quad \text{in
  $\supp \rho^{-,0}_E$}.
\end{align}

Let $\rho^{-,0}$, $\rho^{+,0}$ be two strictly ordered regular initial data and let $f^-, f^+ \in
C(\Rn)$ be two bounded sources such that $\divo \vecb < f^- < f^+ - \e$ for some $\e > 0$. This
strict order will be used in Corollary~\ref{co:order-pressure-contact} to obtain the order of the
limit pressures. We will denote $F^\pm := f^\pm -\divo \vecb$.
Note that
\begin{align*}
  0 < F^- < F^+ - \e.
\end{align*}
Let us define the solutions $\rho^\pm_m$, $i = 1, 2$, of \eqref{pme}
with the respective initial
data $\rho^{\pm,0}$ and sources $f^\pm$, and let $p^\pm_m = P_m(\rho^\pm_m)$ be the pressure solutions.
We can in fact let $\rho^\pm_m$ take on any compactly supported $L^\infty$ data $\rho^{\pm,0}_m$ such
that $\rho^{-,0} = \limsup^* \rho^{-,0}_m$ and $(\rho^{+,0})_* = \liminf_* \rho^{+,0}_m$.
We define the limits
\begin{align*}
  \rho^- &:= \halflimsup_{m \to\infty} \rho^-_m, &
  \rho^+ &:= \halfliminf_{m \to\infty} \rho^+_m,
\intertext{and}
  p^- &:= \halflimsup_{m \to\infty} p^-_m, &
  p^+ &:= \halfliminf_{m \to\infty} p^+_m.
\end{align*}

The main result of this section is the order of the half-relaxed limits for the strictly ordered
initial data, which can be understood as a type of a comparison principle for the limit solutions.
We refer to it as the \emph{almost comparison} for short.
Later, in
Section~\ref{sec:convergence}, we will deduce the full convergence result using the
$L^1$~contraction.

\begin{theorem}[Almost comparison]
  \label{th:almost_cp}
  For strictly ordered regular initial data $\rho^{-,0}$, $\rho^{+,0}$ and
  bounded continuous sources $\divo \vecb \leq f^- < f^+$, we have
  \begin{align*}
    \rho^- \leq \rho^+ \qquad \text{and} \qquad p^- \leq
    p^+ \qquad \text{in $\cl Q := \Rn \times [0, \infty)$},
  \end{align*}
  and
\begin{align*}
  \set{p^- > 0} \subset \set{\rho^- = 1} \subset \set{p^+ > 0} \subset
  \set{\rho^+ = 1}.
\end{align*}
\end{theorem}

The proof of this theorem involves a few technical steps that are developed below. The main tool is
the regularization of solutions by sup- and inf-convolutions.
However, we first derive facts that hold for any half-relaxed limits of \eqref{pme} with regular
initial data.

\subsection{The flow map}
\label{sec:flow-map}

Since \eqref{pme} is a transport-like equation, we introduce the flow map $X: \R \times \Rn \to \Rn$ as
the solution of the ordinary differential equation (ODE)
\begin{align}
  \label{stream}
  \left\{
    \begin{aligned}
      X_t(t, x_0) &= \vec{b}(X(t, x_0)), & & t \in \R,\\
      X(0, x_0) &= x_0,
    \end{aligned}
  \right.
\end{align}
where $X_t$ is the derivative of $X$ with respect to $t$.
As long as $\vec b \in Lip(\Rn)$, we have uniqueness and global existence of $X$, which is continuously differentiable in $t$.
The curves $\set{(X(t, x_0), t): t\in \R}$ are the \emph{characteristics} or \emph{streamlines} of
the flow.
We will also use the notation $X(t): \Rn \to \Rn$ as $X(t)(x) := X(t, x)$.
Uniqueness implies the semigroup property.

\begin{lemma}[Semigroup]
  \label{le:X-semigroup}
  For any $x_0$, $t, s \in \R$ we have
  \begin{align*}
    X(s, x_0) = X(s - t, X(t, x_0))
  \end{align*}
  and therefore
  \begin{align*}
    X(-t, X(t, x_0)) = x_0.
  \end{align*}
  In particular, $X(t): \Rn \to \Rn$ is invertible for $t \in \R$ and  $X(t) ^{-1}= X(-t)$.
\end{lemma}

Furthermore, by Gronwall's inequality, the distance of streamlines decreases at most exponentially.

\begin{lemma}
  \label{le:trajectory-bound}
  Let $L$ be the Lipschitz constant of $\vec b$ on $\Rn$. Then
  \begin{align}
    \label{trajectory-bound}
    e^{-L|t|} |x - y| \leq  |X(t, x) - X(t, y)| \leq e^{L|t|} |x - y| \quad \text{for any $x, y
      \in \Rn$ and $t \in \R$}.
  \end{align}
\end{lemma}

Let $\rho$ be a classical solution of the transport equation \eqref{te}. We have
\begin{align}
  \label{transport-streamline}
  \frac{\partial}{\partial t} \rho(X(t, x_0), t) = \rho_t(X, t) + \vec b(X) \cdot \nabla
  \rho(X, t) = \rho (X(t, x_0), t)(f - \divo \vec b)(X(t, x_0)).
\end{align}
In particular, $t \mapsto \rho(X(t, x_0), t)$ is nondecreasing if $F := f - \divo \vec b \geq 0$.

We define $\rho^\pm_E$ to be the solution of the transport equation \eqref{te} with discontinuous
initial data $\rho^{\pm,0}$ in the following sense:
\begin{align}
  \label{transport-solution}
  \rho^\pm_E(x_0, t_0) := \mu^\pm_{(x_0, t_0)}(0),
\end{align}
where $\mu^\pm_{(x_0, t_0)} = \mu$ is the solution of the simple ODE
\begin{align}
  \label{stream-density-ode}
  \left\{
  \begin{aligned}
    \mu'(t) &= F^\pm (X(t, x_0))\mu(t), &&t \in \R,\\
    \mu(-t_0) &= \rho^{\pm,0}_E(X(-t_0, x_0)).
  \end{aligned}
  \right.
\end{align}
Note that $\rho^\pm_E$ are continuous since $\rho^{\pm,0}_E$ are.

\subsection{Properties of the half-relaxed limits}
\label{sec:prop-half-limits}

Let us first summarize basic properties of $\rho^\pm$ and $p^\pm$, assuming only that they are the
limits of solutions of \eqref{pme} with regular data.

\begin{lemma}
  \label{le:lim_props}
  The following facts about the half-relaxed limits can be derived for any regular initial data.
  The limits $\rho^\pm$ and $p^\pm$ exist and:
  \begin{enumerate}\alphlist
    \item $\rho^-$, $p^-$ are USC while $\rho^+$, $p^+$ are LSC.  Moreover $-\Delta p^-(\cdot, t)
      \leq F^-$ in $\Rn$ for all $t > 0$ while $-\Delta p^+(\cdot, t) \geq
      F^+$ in $\set{p^+(\cdot, t) > 0}$ for all $t > 0$ in the viscosity sense.\\
    \item If $x \in \partial \set{p^-(\cdot, t) > 0}$, $t > 0$, with an exterior ball property at
      $x$, then $p^-(x,t) = 0$.\\
    \item $\rho^- \leq 1$, $\rho^+ \leq 1$.\\
    \item $\set{\rho^- = 1}$ is closed, $\set{p^+ > 0}$ is open in $\set{t \geq 0}$, and so are their time slices.\\
    \item $\set{p^\pm > 0} \subset \set{\rho^\pm = 1}$, $i = 1,2$.
  \end{enumerate}
\end{lemma}

\begin{proof}
  The limits exist by Lemma~\ref{le:pressure_uniform_bound} since the initial data is compactly
  supported.
  (a) follows from the properties of the half-relaxed limits and standard viscosity solution
  arguments, see for instance \cite{K03}.
  Since $p^-$ is bounded by Lemma~\ref{le:pressure_uniform_bound}, (b) is a consequence of a comparison with a radially symmetric solution
  of $-\Delta \phi = \sup F^-$ at $x$.
  Bound on $p^-$, $p^+$ yields (c) as well.
  By (c), $\set{\rho^- = 1} = \set{\rho^- \geq 1}$, which is closed since $\rho^-$ is USC.
  $\set{p^+ > 0}$ is open since $p^+$ is LSC.
  The fact that $P_m^{-1}(s_m) \to 1$ whenever $s_m \to s_\infty > 0$ implies (e).
\end{proof}

A crucial observation is the monotonicity of the set $\set{p^+ > 0}$ along the streamlines.

\begin{lemma}\label{le:char}
  $\set{p^+ > 0}$ is nondecreasing along the streamlines. More
  precisely,
  if $(x_0,t_0) \in \set{p^+ > 0}$ for some $t_0\geq 0$, then $(X(t- t_0, x_0), t) \in \set{p^+ > 0}$ for all $t > t_0$.
 \end{lemma}

\begin{proof}
  We construct a barrier along the streamline.
  Note that $p^+$ is $LSC$ and therefore $\set{p^+ > 0}$ is relatively open in the half-space $\set{t
  \geq 0}$. Let $(x_0, t_0) \in \set{p^+ > 0}$. There exist $r, \lambda, m_0 > 0$ such that $C := \cl
  B_r(x_0, t_0) \cap \set{t \geq 0} \subset \set{p^+ > 0}$, $p^+ > 2 \lambda$ in $C$, and $p^+_m >
  \lambda$ in $C$ for $m \geq m_0$.
  Then by Lemma~\ref{le:pressure_flow_barrier} for every $T > 0$
  there exists $0 < \mu < \lambda$ such that $\pi$ defined in \eqref{pressure-streamline-barrier} is
  a subsolution of \eqref{pmepressure} for $t \in (0,T)$. By comparison, $p^+_m \geq \pi(\cdot,
  \cdot - t_0)$ for all $m \geq m_0$. In particular, $p^+(X(t - t_0, x_0), t) > 0$ for all $t > t_0$.
\end{proof}

We can also compare the half-relaxed limits $\rho^-$ and $\rho^+$ with the solution of the transport equation.
However, since we do not have monotonicity in time, it is not obvious how to prove
Lemma~\ref{le:char} for $\{\rho^- = 1\}$, thus at the moment we cannot
directly show that $\rho^- \leq \rho^-_E$ using the argument of \cite[Lemma~4.4]{KP}.  In the
next section we will show this upper bound only in specific situations
(Lemma~\ref{le:weak_mon_r}), just enough to prove our convergence result through
comparison for perturbed solutions with strictly ordered data, and then use the $L^1$~contraction.
However, if we know that $\rho^- < 1$ at a point and in the past along the streamline up to the
initial time, we can bound $\rho^-$ above by a barrier.

\begin{lemma}
  \label{le:liminf-lower-bound}
  For every $(y, s)$, $s > 0$, we have
  \begin{align*}
    \rho^+(X(t, y), s + t) \geq \min(1, \mu(t)), \qquad t \geq 0,
  \end{align*}
  where $\mu$ is the solution of
  \begin{align}
    \label{mu-eq}
    \left\{
      \begin{aligned}
        \mu'(t) &= F^+ (X(t, y)) \mu(t),\\
        \mu(0) &= \rho^+(y,s).
      \end{aligned}
    \right.
  \end{align}
  Assume additionally that $\rho^-(X(t,y), s + t) < 1$ for $0 \leq t \leq t_0$ for some $y\in\Rn$ and $t_0 > 0$.
  Then
  \begin{align*}
    \rho^-(X(t,y), s + t) \leq \mu(t), \qquad 0 \leq t \leq t_0,
  \end{align*}
  where $\mu(t)$ solves \eqref{mu-eq} with $F^-$ and the initial condition $\mu(0) = \rho^-(y, s)$.

\end{lemma}

\begin{proof}
  Indeed, pick one such point $(y, s)$. Since $\rho^+$ is LSC, for every $\e > 0$ there exists $r > 0$
  with $\rho^+ > \rho^+(y,s) - \e$ on $B_r(y, s)$.
  We use the barrier $\psi_\e$ of the form \eqref{streamline-barrier} from
  Lemma~\ref{le:flow_barrier} where $\eta(x) = (1 - |x|^2)_+$ and $\mu = \mu_\e(t)$ is for some
  fixed $\e \in (0, \rho^+(y, s))$ the solution of
  \begin{align*}
    \left\{
      \begin{aligned}
      \mu_\e'(t) &= \pth{\pth{1 - \exp\pth{- \frac{1 - \e - \mu_\e(t)}\e}}F^+(X(t,
      y)) - 2 \e} \mu_\e(t),\\
      \mu_\e(0) &= \rho^+(y,s) - \e,
      \end{aligned}
    \right.
  \end{align*}
  For given $\e$ we need to take sufficiently small $r > 0$ in the definition of $\psi_\e$.
  Since $\mu \leq 1 - \e$ by the comparison principle, we have
  $\rho^+ \geq \psi_\e$ on the whole space.
  In particular, $\rho^+(X(t, y), s + t) \geq \mu_\e(t)$ for $t \geq 0$.
  Sending $\e \to 0$ implies the claim for $\rho^+$ since $\mu_\e \to \min(1, \mu)$ as $\e \to 0$.

  The proof for $\rho^-$ is parallel.
  By compactness, local uniform continuity of $X$ and upper semi-continuity of $\rho^-$, we
  can find $\delta > 0$ such that $\rho^- < 1 - \delta$ on
  \[
  \mathcal N:= \set{(x, t): 0 \leq t \leq
  t_0,\ |x - X(t, y)| \leq \delta}.
  \]
  For fixed $\e > 0$, we consider barrier $\psi_\e$ of the form \eqref{streamline-barrier}, where
  $\eta(x) = \eta_\e(x) = 1 + \frac{|x|^2}\e$ and
  $\mu_\e'(t) = (F^- + \e)(X(t, y)) \mu_\e(t)$, $\mu_\e(0) = \rho^-(y,
  s) + \e$. Since $\mu_\e \geq \e$, if $r \in (0, \delta)$ in the definition
  of $\psi_\e$,  we have
  $\psi_\e(x, t) > 1$ for $|x - X(t, y)| \geq \delta$. Taking $r$ small enough, by lower
  semi-continuity of $\rho^-$ at $(y, s)$,  we have
  $\psi_\e(x,0) > \rho^-(x, s) \hbox{ for all } x.$

  By Lemma~\ref{le:flow_barrier}, there exists a constant $m_0 = m_0(\e, \delta, r, L, t_0)$ such
  that for all $m > m_0$, $\psi$ is a classical supersolution of the \eqref{pme} at all
  points with $t \in [0, t_0]$ where $\psi_\e < 1
  -\delta$ since $\psi_\e(x,t) > 1$ if $|x - X(t,y)| \geq r e^{-Lt}$. We can moreover
  assume that $\rho^-_m < 1 - \delta$ on $\mathcal N$ by the properties of the half-relaxed limits for all
  $m > m_0$.

  We therefore conclude by the classical comparison that $\rho^-_m(\cdot, s + t) \leq
  \psi_\e(\cdot, t)$ for all $0 \leq t \leq t_0$, $m >
  m_0$. This implies $\rho^-(X(t, y), s + t)
  \leq \mu_\e(t)$ for $0 \leq t \leq t_0$.
  However, $\mu_\e \to \mu$ as $\e \to 0$. We obtain the upper bound for $\rho^-$ by sending
  $\e \to 0$.
\end{proof}

\begin{corollary}
  \label{co:transport-comparison}
  The limit $\rho^+$ satisfies
  \[\rho^+ \geq \min(1,(\rho^+_{\rm tr})_*),\]
  where
  \begin{align}
    \label{rhoT}
    \rho^\pm_{\rm tr}(x,t) := \max(\chi_{\Omega^{\pm,0}}(X(-t,x)), \rho^\pm_E(x,t)).
  \end{align}

Similarly, if $(x_0, t_0)$ is such that $\rho^-(X(t - t_0, x_0), t) < 1$ for $0 \leq t \leq t_0$ then
\begin{align*}
  \rho^-(x_0, t_0) \leq \rho^-_{\rm tr}(x_0, t_0).
\end{align*}
\end{corollary}

\begin{proof}
  We have $(\rho^+_{\rm tr})_*(\cdot, 0) = (\rho^{+,0})_* = \liminf_* \rho^{+,0}_m$ by the definition of
  $\rho^+_{\rm tr}$ and assumptions on the initial data. Therefore for any $y \in \Rn$, $\e \in (0,
  (\rho^{+,0})_*(y))$ there is $r > 0$ and $m_0$ such that $\rho^+_m(\cdot, 0) > (\rho^{+,0})_*(y) -
  \e$ on $B_r(y)$ for $m \geq m_0$. In particular, the barrier in the proof of
  Lemma~\ref{le:liminf-lower-bound} applies with $\mu_\e(0) = (\rho^{+,0})_*(y)  - \e$ and we
  conclude in the limit $\e \to 0$.

  A parallel reasoning together with the barrier in the proof of Lemma~\ref{le:liminf-lower-bound}
applies to $\rho^-$.
\end{proof}

\begin{lemma}
  \label{le:liminf-interior-dense}
  Suppose that $\inf F^+ > 0$. We have
  \begin{enumerate}
    \item[(a)]
      $\set{\rho^+ = 1} \subset \cl{\interior \set{\rho^+ = 1}}$,\\
    \item[(b)]
      $\interior\set{\rho^+ = 1} \subset \set{p^+ > 0}$, and\\
  \item[(c)]
    $\{\rho^+_E \geq 1\} \subset \set{\rho^+ = 1} \subset \cl{\set{p^+ > 0}}$.\\
    \item[(d)] $\cl{\{\rho^+=1\}} = \cl{\set{p^+>0}}$.
      \end{enumerate}
\end{lemma}

\begin{proof}
  Let us write $c:= \inf F^+ > 0$.
 We first show (a). Heuristically, since $\rho^+$ is LSC, if it is $1$ at a point, it is close to $1$ in a
  neighborhood, and therefore the compressive flow/source $F^+ \geq c > 0$ will bring it to $1$ on a set with
  nonempty interior in an arbitrarily small time.  We thus claim that
  \[ (X(t, y), s +
  t) \in \interior \set{\rho^+ = 1} \hbox{ for }t > 0, \hbox{ whenever} \rho^+(y,s)=1.
  \]
   Indeed, by LSC for any $\e > 0$ there exists $r > 0$
  such that $\rho^+ > 1 - \e$ on $B_r(y, s)$. By Lemma~\ref{le:liminf-lower-bound}, we conclude
  that
  \[
  \rho^+ = 1\hbox{ on } \bigcup_{t > - \frac 1c
  \ln (1 - \e)}X(t)(B_r(y)) \times \set{s + t}
  \]
  as  $\mu(t) \geq (1 - \e) e^{ct}$.Therefore we see that every point $(y,s) \in \set{\rho^+ = 1}$ is a limit point of
  $\interior \set{\rho^+ = 1}$. The conclusion follows.

  \medskip

  To show (b), we compare $p^+_m$ with a fast rising subsolution of \eqref{pmepressure}.
  Suppose that $\rho^+ = 1$ on $A := \cl B_{2r}(y) \times [s-\delta, s + \delta]$, $s - \delta \geq
  0$. Let us set $T = 2\delta$, $x_0 = X(-\delta, y)$ and let $N$ be the set
  from Lemma~\ref{le:pressure_flow_barrier}. By Lemma~\ref{le:trajectory-bound}, for $\delta$ sufficiently
  small we have $N + (0, s - \delta) \subset A$.
  Let $L$ be the Lipschitz constant of $\vecb$.
  Pick $q \in (0, 1)$ such that $\frac{c\delta}2 + \log q > 0$.
  For $m$ sufficiently large, $\rho^+_m > q$ and so $p^+_m > q^{m-1}$ on $A$.
  Define
  \begin{align*}
    \mu_m(t) = \min(q^{m-1} e^{(m-1) ct / 2}, \frac{c r^2}{4n} e^{-4L\delta}).
  \end{align*}
  Then by Lemma~\ref{le:pressure_flow_barrier} and the comparison principle for \eqref{pme}, we
  conclude that
  \begin{align*}
  p^+_m \geq \mu_m(t - s + \delta) (1 - r^{-2}|x - X(t - s + \delta, y)|^2 e^{2L(t - s +
  \delta)}) \qquad \text{for $m$ large.}
  \end{align*}
  Since $q^{m-1} e^{(m-1) ct /2} \to \infty$ as $m\to \infty$ for $t$ in a neighborhood of
  $\delta$, we conclude that
  \[
  p^+(y,s) \geq \frac{c r^2}{4n} e^{-4L\delta}.
  \]

\medskip

For (c), the first inclusion is trivial from Corollary~\ref{co:transport-comparison}. The
last
  inclusion is clear from the fact
   \[
   \set{\rho^+ = 1} \subset \cl{\interior\set{\rho^+ = 1}} \subset
  \cl{\set{p^+ = 1}},
  \] where
  the first inclusion is from (a) and the second one is from (b).

\medskip

Lastly,  (d) is a consequence of (c) and Lemma~\ref{le:lim_props} (e).
\end{proof}

\begin{remark}
  \label{re:p-nondegeneracy}
  The proof of Lemma~\ref{le:liminf-interior-dense}(b) implies that
  \[
  p^+(y,s) \geq \frac{c
  r^2}{4n} \quad \hbox{ whenever } B_{2r}(y) \times (s - \delta, s + \delta) \subset \set{\rho^+ = 1} \hbox{
  for some } \delta > 0.
  \]
  In particular, $p^+$ must jump at every time when $\rho^+$ reaches $1$ on a
  set with an nonempty interior.
\end{remark}

\medskip

The next important ingredient is the identification of the initial data for the limits $\rho^-$,
$\rho^+$. Recall that $\rho^{{\pm},0}$ are regular initial data as defined by \eqref{initial}, which in particular implies that these functions are upper semi-continuous.

\begin{lemma}
  \label{le:order-initial}
  The limits $\rho^-$ and $\rho^+$ have the correct initial data, that is,
  \[
  \rho^-(\cdot, 0) = \rho^{-,0}\hbox{ and }\rho^+(\cdot, 0) = (\rho^{+,0})_*.
  \]
\end{lemma}

\begin{proof}
  By definition of half-relaxed limits, we have automatically $\rho^-(\cdot, 0) \geq \rho^{-,0}$ and
  $\rho^+(\cdot, 0) \leq (\rho^{+,0})_*$ since $\rho^{-,0}$ is USC and $(\rho^{+,0})_*$ is LSC by construction.
  We therefore only need to prove the opposite inequalities.

\medskip

  The inequality for $\rho^+$ follows from Corollary~\ref{co:transport-comparison} since
  $(\rho^+_{\rm tr})_*$ is LSC
  and equal to $\rho^{+,0}$ at $t = 0$.

\medskip

For $\rho^-$ the situation is a little bit more involved. First we note that
\[
\rho^-(\cdot, 0) \leq \sup
\rho^{-,0} \leq 1
\] by Lemma~\ref{le:lim_props}(c).
Now choose $x_0$ such that $\rho^{-,0}(x_0) < 1$.
For any $\mu \in (\rho^{-,0}(x_0), 1)$, there is $\eta > 0$ such that
\[
\rho^{-,0}(x) < \mu < 1\hbox{ for } |x - x_0|
\leq \eta.
\]
Let us take $M := \sup_m \norm{p_m}_\infty,$ which is finite due to Lemma~\ref{le:pressure_uniform_bound}.
We can now construct a radially symmetric contracting classical solution of \eqref{pmeKP} on a
domain $B_\eta \times (-\tau, \tau)$ initial data $\Omega_0 = B_\eta \setminus \cl B_{\frac\eta2}$,
boundary data $p = 2M$ on $\partial B_\eta$, and initial data for the external density
$\rho^E(\cdot, 0) = \mu$. By the construction in Section~\ref{sec:barriers-rho-one} this yields superbarriers for
$\rho^-_m$ in $B_{\frac \eta2}(x_0) \times [0, \delta)$, $\delta > 0$ small independent of $\mu$.
We conclude that
\[
\rho^-(x_0, 0) \leq \mu.
\] The claim follows.
\end{proof}

\begin{lemma}
  \label{le:up-positive-init}
  $p^+(\cdot,0)>0$ in $\interior \Omega^{+,0}$, provided that $\inf F^+ > 0$.
\end{lemma}

\begin{proof}
  If $x_0 \in \interior \Omega^{+,0}$, then $p^+_m(x_0, 0) =
  P_m(\rho^+_m(x_0, 0)) = \frac m{m-1} > 1$.
  Moreover, there is $r > 0$ such that
  $B_{2r}(x_0) \subset \interior \Omega^{+,0}$.
  Let us set $\mu := \frac{r^2}{8n} \inf F^+ >0$. We may assume that $\mu < 1$. Then, for some
  small $T >0$, independent of $m$, the pressure subsolution
  $\pi$ defined in \eqref{pressure-streamline-barrier} stays under
  $p^+_m$ for time $t \in [0, T]$. Hence $p^+(x_0,0) > 0$.
\end{proof}

Let us also state the ``left accessibility'' of $\set{\rho^- = 1}$, i.e., that $\rho^-$ cannot jump
up to $1$.
\begin{lemma}
  \label{le:left-accessible}
  The set $\set{\rho^- = 1}$ is ``left-accessible'', that is, for all $(x_0, t_0) \in \set{\rho^- =
  1}$, $t_0 > 0$,
  there exists a sequence $(x_k, t_k) \to (x_0, t_0)$ such that $t_k < t_0$ and $\rho^-(x_k,
  t_k) \to 1$ as $k\to\infty$.
\end{lemma}

\begin{proof}
  Suppose that for a fixed point $(x_0, t_0)$, $t_0 > 0$, there does not exist the claimed sequence, and thus
  we can find $\delta > 0$ such that $\rho^- < 1 - \delta$ on $B_\delta(x_0) \times [t_0 - \delta,
  t_0)$.
  A comparison with a superbarrier from
  Section~\ref{sec:barriers-rho-one} implies $\rho^-(x_0, t_0) \leq 1 - \delta$ and hence $(x_0,
t_0) \notin \set{\rho^- = 1}$.
\end{proof}

The strict ordering of the initial data yields a strict order of the external densities.
Recall the definition of $\rho^\pm_E$ from Section~\ref{sec:flow-map}.

\begin{lemma}
  \label{le:rho-ext-order}
  For every $T > 0$ there exists $\delta>0$ such that the following holds for $|(x_1, t_1) - (x_2, t_2)| < \delta$ and
  $0\leq t_1, t_2 \leq T$:
  \begin{enumerate}
    \item[(a)] $\rho^-_E(x_1, t_1) \leq \rho^+_E(x_2,t_2).$\\
    \item[(b)] If $\rho^+_E(x_2, t_2) \leq 1$,  then  $\rho^-_E(x_1, t_1) < 1- \delta$.\\
  \item[(c)]  If $ p^-(x_1,t_1)>0,$  then  $F^-(x_1, t_1) < F^+(x_2, t_2).$\\
  \item[(d)]  If $(-t_2, x_2) \notin \interior \Omega^{+,0}$ then $X(-t_1, x_1) \notin \Omega^{-,0}.$
  \end{enumerate}
\end{lemma}

\begin{proof}
  This follows from the strict order of the initial data \eqref{initial-data-order}, compactness and uniform
  continuity.

  Indeed, by \eqref{initial-data-order}, we have $\rho^-_E <
  \rho^+_E$ in $\supp \rho^-_E$.
  By compactness and hence uniform continuity, there is $\delta > 0$ such that $\rho^-(x_1, t_1) <
  \rho^+(x_2, t_2) - \delta$ for all $(x_1, t_1) \in \supp \rho^-_E$, $|(x_1,
  t_1) - (x_2, t_2)| <\delta$, $t_1, t_2 \in [0, T]$. In particular, $\rho^-_E(x_1,t_1) <
  1-\delta$ if $\rho^+_E(x_2, t_2)\leq 1$. This yields (a) and (b).

  Similarly, $\set{p^- > 0} \cap \set{0 \leq t \leq T}$ is bounded, and so $F^-$ and $F^+$ are
  uniformly continuous in its neighborhood, and therefore if we take $\delta > 0$ small enough, we
  have $F^-(x_1, t_1) < F^+(x_2, t_2)$ if $p^-(x_1, t_1) > 0$, yielding (c).

  (d) follows from Lemma~\ref{le:trajectory-bound} and \eqref{initial-data-order}.
\end{proof}

\subsection{Proof of the order \texorpdfstring{$\set{\rho^- = 1} \subset \set{p^+ > 0}$}{rho1 = 1
subset p2 greater than 0}}
\label{sec:proof-of-support-order}

We prove the order of $\rho^-$ and $\rho^+$ by a barrier argument using the knowledge of the
convergence in the radial case for the equation \eqref{pme} without drift \cite{KP}.
Our argument follows the proof of the comparison principle, showing that if $\set{\rho^- = 1}$ and
$\set{p^+ > 0}$ fail to stay ordered, that is, their boundaries
have a certain contact, we get a contradiction. To be able to argue at the contact point, we need to find a
contact point with a suitable regularity, essentially $C^{1,1}$ in space. This is where sup- and inf-convolutions come in.

\medskip

Let $\Xi_r$ be the ``flattened'' set of size $r > 0$,
\begin{align*}
  \Xi_r := \set{(x,t): \max(|x| - r, 0)^2 + t^2 < r^2}, \qquad \Xi_r(x,t) := \Xi_r + (x,t).
\end{align*}
For fixed $r_0 > 0$, we define the time-dependent size
\begin{align}
  \label{rt}
  r(t) := r_0 e^{-2Lt},
\end{align}
where $L$ is the Lipschitz constant of $\vec b$,
and the sup-/inf-convolutions
\begin{align}
  \label{sup-inf-convolutions}
  \begin{aligned}
    \rho^{-,r}(x, t) := \sup_{\cl \Xi_{r(t)}(x, t)} \rho^-, \quad
    \rho^{+,r}(x, t) := \inf_{\cl \Xi_{r(t)}(x, t)} \rho^+.
  \end{aligned}
\end{align}
Analogously, we introduce $p^{\pm, r}$, $\rho^{\pm, r}_E$ and $F^{\pm,r}$.
These functions are well defined on $\set{t \geq \tau}$ for $\tau > 0$ the unique solution of
$r(\tau) = \tau$. We will understand notation like $\set{\rho^{-,r} = 1}$ as $\set{(x,t): t \geq
\tau,\ \rho^{-,r}(x, t) = 1}$. If we speak of such sets as open or closed, we always mean in the
relative topology of $\set{(x,t): t \geq \tau}$.
Note that convolving over $\cl \Xi_r$ is equivalent to doing two successive convolutions over a
closed space ball of radius $r$ and a closed space-time ball of radius $r$.
Also note that sup-convolution commutes with $\halflimsup$ and
inf-convolution commutes with $\halfliminf$.

\medskip

The purpose of varying the size of the set over which we convolve has to do with the spreading of
the streamlines.  If the boundaries of $\set{\rho^{-,r} = 1}$ and $\set{p^{+,r} > 0}$ cross, the normal
velocity of $\partial\set{\rho^- = 1}$ must be larger than the normal velocity of
$\partial\set{p^+
> 0}$ by an amount that is enough to absorb the difference of the velocity of streamlines within
$\Xi_{r(t)}(x,t)$ caused by the dependence of $\vec b$ on $x$, see \eqref{velocity-gap} below.
Furthermore, we want the convolutions to be monotone along the streamlines, and for that we need to
ensure that the streamlines intersecting $\Xi_{r(t)}(x, t)$ also intersect $\Xi_{r(t - s)}
(X(-s, x), t - s)$ for all $s > 0$. This is guaranteed by the choice of $r(t)$ in \eqref{rt} and by
Lemma~\ref{le:trajectory-bound}.  This reasoning implies the following lemma.

\begin{lemma}\label{le:char_r}
  $\set{p^{+,r} > 0}$ is nondecreasing along the streamlines as in Lemma~\ref{le:char}.
 \end{lemma}

\medskip

Let us introduce the \emph{first contact time} as
\begin{align}
  \label{contact-time-r}
  t_0:= \sup\set{s: \set{\rho^{-,r} = 1} \cap \set{t \leq s} \subset \set{p^{+,r}>0}}.
\end{align}

\begin{claim}
  \label{cl:no-contact}
  For every $T > 0$ there exists $\tilde r_0$ such that for all $r_0 \in (0, \tilde r_0)$ in
  \eqref{rt} we have
  $t_0 \geq T$.
\end{claim}

To establish the claim, let us fix $T > 0$ and choose $\tilde r_0 > 0$ so that
Lemma~\ref{le:rho-ext-order} applies for some fixed $\delta > 0$ for all $(x_1, t_1), (x_2, t_2) \in
\Xi_{\tilde r_0}(x,t)$ for all $0 \leq t \leq T$, $x \in \Rn$.
Suppose that $t_0 < T$.

By the initial strict separation, we have $t_0 > \tau$.

\begin{lemma}
  \label{le:t0_pos}
  There exists $\delta > 0$ such that for any $r_0 \in (0, \delta)$ we have $t_0 > \tau$, where $\tau =
\tau(r_0)$ is the unique solution of $r(\tau) = \tau$.
\end{lemma}

\begin{proof}
  This is a consequence of the fact that $\set{\rho^- = 1}$ and $\set{p^+ = 0}$
are closed, and their time slices at $t = 0$ are disjoint and $\set{\rho^- = 1,\ 0 \leq t \leq s}$
is compact for any $s > 0$.
  Indeed, we have
  \begin{align*}
    \set{\rho^-(\cdot, 0) = 1} = \Omega^{-,0} \subset \interior \Omega^{+,0} \subset
    \set{p^+(\cdot, 0) > 0},
  \end{align*}
  where the first equality follows from
  Lemma~\ref{le:order-initial}, the first inclusion is from \eqref{initial-data-order}, and the
  last inclusion is given in Lemma~\ref{le:up-positive-init}.
\end{proof}

The geometry of $\Xi_{r(t)}$ and the definition of $t_0$ has the following consequence on the
relationship of sets $\set{\rho^- = 1}$ and $\set{p^+ > 0}$.

\begin{lemma}
  \label{le:geometry}
  Let $(y_1, s_1)$, $(y_2, s_2)$ be any two points such that $\rho^-(y_1, s_1) = 1$ and $p^+(y_2,
  s_2) = 0$. Then
  \begin{align*}
    |y_1 - y_2| \geq \zeta(t; s_1, s_2) :=
    \begin{cases}
      \sum_{i=1,2} (r(t)^2 - (t- s_i)^2)^\frac12 + r(t), & \max_{i=1,2}{|t - s_i|} < r(t),\\
      0, & \text{otherwise},
    \end{cases}
  \end{align*}
  for all $\tau \leq t \leq t_0$.
  Note that $\zeta(t; s_1, s_2)$ is the sum of radii of the space slices of $\Xi_{r(t)}$ at times $s_1$
  and $s_2$, and therefore if the inequality above is violated, both $(y_1, s_1)$ and $(y_2, s_2)$
  lie in $\Xi_{r(t)}(x,t)$ for some $x$.
\end{lemma}

\begin{proof}
  The definition of $t_0$ \eqref{contact-time-r} implies
  \begin{align}
    \label{empty-interior-intersect}
    \set{\rho^{-,r} = 1} \cap \set{p^{+,r} = 0} \cap \set{\tau \leq t < t_0}  = \emptyset.
  \end{align}
  Suppose that we have two points $(y_1, s_1)$ and $(y_2, s_2)$ satisfying the hypothesis, but $|y_1
  - y_2| < \zeta(t; s_1, s_2)$ for some $\tau \leq t \leq t_0$. In particular, $\max_i |t - s_i| < r(t)$. Thus by
  continuity and Lemma~\ref{le:t0_pos}, we can make $t$ smaller if necessary and assume that $\tau
  \leq t < t_0$. Then there exists a point $x$ such that
  \begin{align*}
    |x- y_i| < (r(t)^2 - (t- s_i)^2)^\frac12 + r(t), \qquad i = 1,2.
  \end{align*}
  In particular, $(y_i, s_i) \in \Xi_{r(t)} (x,t)$, $i = 1,2$. By definition of the sup-/inf-convolutions, we
  have $(x,t) \in \set{\rho^{-,r} = 1} \cap \set{p^{+,r} = 0}$, a contradiction
  with \eqref{empty-interior-intersect}.
\end{proof}

Now we are able to clarify the validity of $\rho^- \leq \rho^-_{\rm tr}$ in
Corollary~\ref{co:transport-comparison} in terms of $\set{p^+ > 0}$.

\begin{lemma}\label{le:weak_mon_r}
  We have
\begin{equation*}
  \rho^- \leq \rho^-_{\rm tr} \qquad \text{on }\Xi_{r(t)}(x,t) \text{ for any }(x,t) \in
\{p^{+,r} = 0 \} \cap \{t \leq t_0\},
\end{equation*}
where $\rho^-_{\rm tr}$ was defined in \eqref{rhoT}.
\end{lemma}

\begin{proof}
  Let us set $U := \set{p^{+,r} = 0} \cap \set{t \leq t_0}$.
  The key step is to show that $\rho^- < 1$
  backwards in time along all streamlines starting in $\Xi_{r(t)} (x, t)$ for any $(x,t) \in
  U$. Then we can use Corollary~\ref{co:transport-comparison}.

  Let $(x,t) \in U$ and $(y_1, s_1) \in \Xi_{r(t)}(x,t)$. By the continuity of $r(t)$ and by
  Lemma~\ref{le:char_r}, there exists $(y_0, s_0) \in U$ with $s_0 < t_0$ such that $(y_1, s_1) \in
  \Xi_{r(s_0)}(y_0, s_0)$.
  As $p^{+,r}(y_0, s_0) = 0$, there also exists $(y_2, s_2) \in \cl \Xi_{r(s_0)}(y_0, s_0)$ such
  that $p^+(y_2,s_2) = 0$. By Lemma~\ref{le:char}, $p^+ = 0$ backwards in time along the streamline
  going through $(y_2, s_2)$.

  Recall the definition of $\tau > 0$ as the unique soluiton $r(\tau) = \tau$.
  By the definition of $r(t)$ and $\Xi_{r(t)}$, and the estimate \eqref{trajectory-bound},
  we have
  \begin{align*}
  (X(-\sigma, y_i), s_i - \sigma) \in \cl \Xi_{r(s_0 - \sigma)}(X(-\sigma, y_0), s_0 -
  \sigma) \qquad \text{for } i = 1,2,\ \sigma > 0.
  \end{align*}
  Therefore
  \begin{align*}
    p^{+,r}(X(-\sigma, y_0), s_0 - \sigma) \leq p^+(X(-\sigma, y_2), s_2 - \sigma) = 0 \qquad
    \text{for $s_0 - \sigma \geq \tau$},
  \end{align*}
  which implies by definition of $t_0$
  \begin{align*}
    \rho^-(X(-\sigma, y_1), s_1 - \sigma) \leq \rho^{-,r}(X(-\sigma, y_0), s_0 - \sigma)  < 1
    \qquad \text{for $s_0 - \sigma \geq \tau$}.
  \end{align*}
  When $s_0 - \sigma < \tau$, and specifically when $s_1 - \sigma \geq 0$ is close to the initial
  time $0$,
  we need to argue as in the proof of Lemma~3.16 to deduce $\rho^-(X(-\sigma, y_1), s_1 - \sigma) <
  1$.
  Therefore $\rho^- \leq \rho^-_{\rm tr}$ at $(y_1, s_1)$ by Corollary~\ref{co:transport-comparison}.
\end{proof}

Recall that to establish Claim~\ref{cl:no-contact}, we suppose that $t_0 < T$. We first observe that then there is a
``nice'' contact point $(x_0, t_0)$.

\begin{lemma}
  \label{le:contact_point}
  If $t_0 < \infty$ then:
  \begin{enumerate}\alphlist
    \item $\set{\rho^{-,r}(\cdot, t_0) = 1} \subset \cl{\set{p^{+,r}(\cdot, t_0) > 0}}$
    \item There exists a point $x_0$ with $\rho^{-,r}(x_0, t_0) = 1$ while $p^{+,r}(x_0, t_0) = 0$.
    \item At every contact point $x_0$, $p^{-,r}(x_0, t_0) = 0$.
    \item For every contact point $x_0$, there exist points $(x_1, t_1)$ and $(x_2, t_2)$ in
      $\partial \Xi_{r(t_0)}(x_0, t_0)$ such that $\rho^-(x_1, t_1) = 1$, $p^+(x_2, t_2) = 0$ and
      $\abs{t_i - t_0} < r(t_0)$, $i = 1,2$ (finite free boundary speed).  Moreover, $p^-(x_1, t_1) = 0$.
  \end{enumerate}
\end{lemma}

\begin{proof}[Proof of Lemma~\ref{le:contact_point}(a)]
  This is a consequence of the facts that $\rho^-$ cannot jump up to $1$ and $p^+$ cannot jump down
  to 0. More rigorously,
  let $\delta > 0$ be the constant given from Lemma~\ref{le:rho-ext-order}.  Suppose that
  \[
  y_0
  \notin \cl{\set{p^{+,r}(\cdot, t_0) > 0}}.
  \]  Take any sequence $(y_k, t_k) \to (y_0, t_0)$ with
  $t_k < t_0$.  By continuity of $X$, for large $k$ we have
  $ X(t_0 - t_k, y_k) \notin
  \cl{\set{p^{+,r}(\cdot, t_0) > 0}}$ and therefore Lemma~\ref{le:char_r} yields
  \[
  (y_k, t_k) \notin \cl{\set{p^{+,r} > 0}}.
  \]  By Lemma~\ref{le:liminf-interior-dense}(c), we have
  \[
  \rho^{+,r}_E(y_k, t_k) < 1.
  \] Therefore $\rho^{-,r}_E(y_k, t_k) < 1- \delta$ by
  Lemma~\ref{le:rho-ext-order}.  In particular, Lemma~\ref{le:weak_mon_r} yields $\rho^{-,r}(y_k,
  t_k) \leq \rho^{-,r}_E(y_k, t_k) < 1 - \delta$. Since the sequence is arbitrary, we conclude that
  $\rho^{-,r}(y_0, t_0) < 1$ by Lemma~\ref{le:left-accessible}.
\end{proof}

\begin{proof}[Proof of Lemma~\ref{le:contact_point}(b)]
  If such point does not exist, we have $\set{\rho^{-,r}(\cdot, t_0) = 1} \subset
  \set{p^{+,r}(\cdot, t_0) > 0}$ and therefore
  \[
  \set{\rho^{-,r} = 1} \cap \set{t \leq t_0} \subset
  \set{p^{+,r} > 0}.
  \]
   But this is a contradiction with a definition of $t_0$ since the
  set of $s$ on the right-hand side of \eqref{contact-time-r} is open by the compactness of the set
  $\set{\rho^{-,r} = 1} \cap \set{t \leq s}$ for all $s$.
\end{proof}

\begin{proof}[Proof of Lemma~\ref{le:contact_point}(c)]
 We can conclude that $p^{-,r}(x_0, t_0) = 0$ from Lemma~\ref{le:lim_props}(a), the fact that $-\Delta p^-(\cdot, t_1) \leq
  F^-$ in
  $\Rn$, and the existence of the exterior ball of $\set{p^{-,r}(\cdot, t_1) > 0}$ at $(x_0, t_0)$
  (interior ball of $\set{p^{+,r}(\cdot, t_0) = 0})$.
\end{proof}

\begin{proof}[Proof of Lemma~\ref{le:contact_point}(d)]
  Let us fix a contact point $(x_0, t_0)$ from (b) and set $r = r(t_0)$.  The existence of points
  $(x_1, t_1)$ and $(x_2, t_2)$ in $\cl\Xi_{r}(x_0,t_0)$ with $\rho^-(x_1,t_1) = 1$ and $p^+(x_2,
  t_2) = 0$ is clear from semi-continuity and the definition of $\rho^{-,r}$,
  $p^{+,r}$.

  By (a), we must have $p^+ > 0$ and $\rho^- < 1$ on $\Xi_r(x_0, t_0)$. In particular,
  $
  (x_i, t_i)   \in \partial \Xi_r(x_0, t_0) \hbox{ for } i = 1,2. $

  Let us show
  that  $t_1,t_2$ stay away from the boundary of $\Xi_r(x_0,t_0)$, i.e.,
  \begin{equation*}
  |t_i - t_0| < r \hbox{ for } i = 1, 2.
  \end{equation*}
 We first observe that $t_2 < t_0 + r$ by the monotonicity of $\set{p^+ > 0}$ along the
  streamlines, Lemma~\ref{le:char}. Indeed, if $t_2 = t_0 + r$, then $|x_2 - x_0| \leq r$ and
  $|X(\sigma, x_2) - x_0| \leq r + (r^2 - (t_0 - t_2 - \sigma)^2)^{1/2}$ for $0 < - \sigma \ll 1$
  by Lemma~\ref{le:trajectory-bound}. Therefore $(X(\sigma, x_2), t_2 + \sigma) \in \Xi_r(x_0,
  t_0)$, yielding $p^+(X(\sigma, x_2), t_2 + \sigma) > 0$, which contradicts
  Lemma~\ref{le:char}.

\medskip

It is more work to prove $t_1 < t_0 + r$. We first notice that $\rho^- \leq \rho^-_{\rm tr}$ in
  $\Xi_r(x_0, t_0)$ by
  Lemma~\ref{le:weak_mon_r}. Furthermore, by definition, if $(\rho^+_{\rm tr})_*(x, t) <1$ then
  $\rho^+_E(x,t) < 1$ and $(x, t) \notin \interior \Omega^{+,0}$.
  Thus, since $p^+(x_2, t_2) = 0$, by the continuity of $\rho^+_E$
  and Lemma~\ref{le:liminf-interior-dense}(c) we must have $\rho^+_E(x_2, t_2) \leq 1$
  and $X(-t_2, x_2) \notin \interior \Omega^{+,0}$.
  But then
  Lemma~\ref{le:rho-ext-order} and the choice of $r_0$ imply that $\rho^-_E < 1 - \delta$ on
  $\cl \Xi_r(x_0, t_0)$ for some $\delta > 0$, and $X(-t, x) \notin \Omega^{-,0}$ for all
  $(x,t) \in \cl \Xi_r(x_0, t_0)$. The latter implies $\rho^-_{\rm tr} = \rho^-_E$ in $\cl \Xi_r(x_0,
  t_0)$.

  Now, by the same argument as in the proof of (a), we conclude that $\rho^- \leq \rho^-_E <
  1 - \delta$, and hence
  \[
  p^- = 0 \hbox{ at } (x,t_0 + r) \hbox{ for } |x - x_0| < r.
  \]
   In particular, if $t_1
  = t_0 + r$, we have $|x_1 - x_0| = r$.
  This means that the boundary of $\set{\rho^-(\cdot, t_1) = 1}$ has an exterior ball of radius $r$
  at $x_1$. By a comparison with a radial supersolution for the elliptic problem, $p^- = 0$ on the boundary of
  $\Xi_r(x_0,t_0)$, and the growth of $p^-$ away from the boundary in space is controlled. However, $\set{\rho^-
  = 1}$ expands with an ``infinite speed'' at $(x_1, t_1)$ into a region with external density $\rho^-_E <
  1- \delta < 1$. An argument as in the proof of Lemma~\ref{le:velocity-bounds} below applies (see
  the details there), and we reach a contradiction.

\medskip

We have shown that $\max_i t_i < t_0 + r$. From Lemma~\ref{le:geometry} we deduce that
  $\min_i t_i > t_0 - r$.
\end{proof}

\begin{corollary}
  \label{co:order-pressure-contact}
  $p^{-,r} \leq p^{+,r}$ for $t \leq t_0$. The order is strict in $\set{p^{-,r} > 0}$.
\end{corollary}

\begin{proof}
  For $t < t_0$ this just follows from the comparison principle for Poisson's equation by
  Lemma~\ref{le:lim_props}(a).
  First note that $p^{-,r}
  = 0$ on $\partial \set{p^{+,r}(\cdot, t) > 0}$ for $t < t_0$. Indeed
  \[
  \set{p^{-,r}(\cdot, t) > 0}
  \subset \set{\rho^{-,r}(\cdot, t) = 1} \subset \set{p^{+,r}(\cdot, t) > 0}.
  \]
  At $t = t_0$ we use Lemma~\ref{le:contact_point}(a) and (c).
  Moreover, by a standard viscosity solution argument, $-\Delta p^{-,r}
  \leq F^{-,r}$ and $-\Delta
  p^{+,r} \geq F^{+,r}$ in the open set $\set{p^{+,r} > 0}$, where $F^{\pm,r}$ are defined as in
  \eqref{sup-inf-convolutions}.  But by Lemma~\ref{le:rho-ext-order} and the choice of $r(t)$, we
  have $F^{-,r} < F^{+,r}$ in $\set{p^{-,r} > 0}$. The elliptic comparison principle yields
  \[p^{-,r} <
  p^{+,r} \hbox{  in }\set{p^{-,r} > 0}.
  \]
\end{proof}

To finish the proof of Claim~\ref{cl:no-contact}, suppose that $t_0 < T$ and choose points
$(x_0, t_0)$, $(x_1, t_1)$, $(x_2, t_2)$ from Lemma~\ref{le:contact_point}. We will construct
barriers for $\rho^-_m$ at $(x_1, t_1)$ and for $\rho^+_m$ at $(x_2, t_2)$ that will yield a
contradiction with the comparison principle for \eqref{pme}. We will use Lemma~\ref{le:geometry} to
keep track of the geometry of the free boundaries around $(x_1, t_1)$ and $(x_2, t_2)$, including the
existence of exterior and interior balls, as well as ordering of normal velocities.

First, since $(x_1, t_1), (x_2, t_2) \in \partial \Xi_{r(t_0)}(x_0, t_0)$ and $\max_i|t_0 - t_i| <
r(t_0)$, we have by
Lemma~\ref{le:geometry}
\begin{align*}
  |x_1 - x_2| \leq \sum_{i=1,2} |x_i - x_0| = \zeta(t_0; t_1, t_2) \leq |x_1 - x_2|.
\end{align*}
This implies that $x_0 - x_1$ and $x_0
- x_2$ are nonzero and parallel with opposite directions. We will set
\begin{align*}
  \nu := \frac{x_2 - x_0}{|x_2 - x_0|}.
\end{align*}
This will serve as the outer unit normal vector of the free boundaries at the contact point.

Moreover, $t \mapsto \zeta(t; t_1, t_2)$ has a maximum at $t_0$ in $t \leq t_0$ and therefore $\partial_tf(t_0;
t_1, t_2) \geq 0$, which, after
rearranging the terms, yields
\begin{align}
  \label{speed-diff-bound}
  \sum_{i=1,2} \frac{t_0 - t_i}{(r(t_0)^2 - (t_0 - t_i)^2)^\frac12} \leq
  r'(t_0) \sum_{i=1,2}  \pth{1 + \frac{r(t_0)}{(r(t_0)^2 - (t_0 - t_i)^2)^\frac12}} \leq 4 r'(t_0),
\end{align}
since $r'(t_0) < 0$.

As a proxy for the normal velocity of $\set{\rho^- = 1}$ at $(x_1, t_2)$ and the normal velocity of
$\set{p^+ > 0}$ at $(x_2, t_2)$, we define
\begin{align*}
  V_1:= -\partial_{s_1} \zeta(t_0; t_1, t_2), \qquad V_2:= \partial_{s_2} \zeta(t_0; t_1, t_2).
\end{align*}
Differentiating, we get the estimate
\begin{align}
  \label{velocity-gap}
  -V_1 + V_2 = \partial_{s_1} \zeta(t_0; t_1, t_2) + \partial_{s_2} \zeta(t_0; t_1, t_2)
  = \sum_{i=1,2} \frac{t_0 - t_i}{(r(t_0)^2 - (t_0 - t_i)^2)^\frac12} \leq 4 r'(t_0),
\end{align}
where we used \eqref{speed-diff-bound} for the last inequality.
Geometrically, in view of Lemma~\ref{le:geometry}, $V_1$ gives a \emph{lower} bound on the normal velocity of $\partial
\set{\rho^- = 1}$
at $(x_1, t_1)$, while $V_2$ gives an \emph{upper} bound on the normal velocity of $\partial
{\set{p^+ > 0}}$ at $(x_2, t_2)$.

We define the (space) slopes of pressure at $(x_0, t_0)$ as
\begin{align}
  \label{eta}
  \eta_1 := \limsup_{h \to 0+} \frac{p^{-,r}(x_0 - h\nu, t_0)}h, \qquad
  \eta_2 := \liminf_{h \to 0+} \frac{p^{+,r}(x_0 - h\nu, t_0)}h.
\end{align}
We have
\begin{align}
  \label{weak-grad-order}
  \eta_1 < \eta_2
\end{align}
by Corollary~\ref{co:order-pressure-contact} and Hopf's lemma.
By definition, we have $\sup_{\cl\Xi_{r(t_0)}(x_0 - h\nu, t_0)} p^- = p^{-,r}(x_0 - h\nu, t_0)$, so
$\eta_1$ provides a bound on the slope of $p^-$ at $(x_1,t_1)$. Recall that $p^-(x_1,t_1) = 0$ by
Lemma~\ref{le:contact_point}(d). An analogous reasoning applies to
  $p^+$ at $(x_2,t_2)$.

\medskip

By a barrier argument, we obtain the following bounds on the normal velocity.

\begin{lemma}
  \label{le:velocity-bounds}
\[
 V_1 \leq \frac{\eta_1}{1 - \rho^-_E(x_1, t_1)} + \vec b(x_1) \cdot \nu \quad \hbox{ and }\quad
V_2 \geq \frac{\eta_2}{1 - \rho^+_E(x_2, t_2)} + \vec b(x_2) \cdot \nu.
\]
\end{lemma}
\begin{proof}
We will use a barrier argument using a barrier in Section~\ref{sec:barriers-rho-one}
at points $(x_1, t_1)$ and $(x_2, t_2)$.
Let us set up the barrier at $(x_1,t_1)$.

Suppose that the inequality for $V_1$ is violated. Then there exists $\e > 0$ such that
\begin{align}
  \label{V_1_violation}
V_1 > \frac{\eta_1 + \e}{1 - \rho^-_E(x_1,t_1) - \e} + \vec b(x_1) \cdot \nu + \e.
\end{align}

Let us define the space radius
\begin{align*}
  s(t) := |x_1 - x_0| + (V_1 - \e)(t_1 - t).
\end{align*}
Let us take $h^* \in (0, s(t_1))$ and set $x^* := x_0 - h^* \nu$.
Define the smooth function $\phi = \phi(x,t)$, radially symmetric with respect to $x^*$, for $x
\neq x^*$ as
\begin{align*}
  \phi(x,t) = (\eta_1 + \e) (|x - x_*| - s(t) + h^*).
\end{align*}
By construction and due to assumption \eqref{V_1_violation}, we have
\begin{align*}
  \frac{\phi_t(x_1,t_1)}{|\nabla \phi|(x_1,t_1)} = V_1 - \e > \frac{|\nabla \phi|(x_1,t_1)}{1 -
  \rho^-_E(x_1,t_1) - \e} + \vecb(x_1) \cdot \nu.
\end{align*}
We set
\[
  \rho^E(x,t) := (\rho^-_E(x_1,t_1) + \e) e^{-(t - t_1) \sup F^-}.
\]
Let us check that
\begin{align}
  \label{p-phi}
  p^- < \phi \qquad \text{on $U \cap \cl{\set{p^- > 0}} \cap \set{t \leq t_1} \setminus
  \set{(x_1,t_1)}$,}
\end{align}
for some small neighborhood $U \ni (x_1, t_1)$.
We define the radius of $\Xi_{r(t_0)}(x_0, t_0)$ as
\begin{align*}
  \zeta_0(t) := r(t_0) + (r(t_0)^2 - (t - t_0)^2)^{1/2}.
\end{align*}
By the definition of $\eta_1$ in \eqref{eta}, for small $h \geq 0$ we have $p_-(x,t) \leq (\eta_1 +
\tfrac \e2)h$ for $x \in \cl B_{\zeta_0(t)}(x_0 - h\nu)$, $\abs{t - t_0} \leq r(t_0)$, with equality only for $h = 0$.
On the other hand, $\phi(x,t) \geq (\eta_1 + \tfrac \e2)h$ on $\partial
B_{s(t) - h^* + h}(x^*)$, with equality only for $h = 0$.
For $t\leq t_2$ close to $t_2$, we have $\zeta_0(t) \geq s(t) + h^*$, with equality only at $t =
t_2$. Therefore
\[
\partial B_{s(t) - h^* + h}(x^*) \subset \cl B_{\zeta_0(t)}(x_0 - h\nu) \hbox{ for } 0 \leq t_2 - t \ll 1, 0 \leq h < h^*.
\] By making $U$ sufficiently small, \eqref{p-phi} follows.

\medskip

By a similar argument, recalling that $\rho^- \leq \rho^-_E$ in $\Xi_{r(t_0)}(x_0,t_0)$, by making $U$ smaller if necessary, we deduce that
\[
\rho^- \leq
\chi_{\set{\phi>0}} + \rho^E \chi_{\set{\phi>0}^\compl} \hbox{ on } U \cap \set{t \leq t_1} \setminus
\set{(x_1,t_1)},
\] with strict inequality when $\rho^- < 1$.  By Proposition~\ref{pr:barrier-sequence}, there exists a sequence $\varphi_\rho^m$, $\varphi_p^m$
of supersolutions of \eqref{pme} in a parabolic neighborhood $\mathcal N$ of $(x_1,t_1)$. By
Lemma~\ref{le:barrier-order}, for large $m$ we must have $p^-_m \leq \varphi_p^m$ and $\rho^-_m
\leq \varphi_\rho^m$, even allowing for small translations of $\varphi_p^m$, $\varphi_\rho^m$.
However this is a contradiction with the fact $\rho^-(x_1,t_1) = 1$ and $(x_1,t_1) \in \partial
\set{\liminf_* \varphi_\rho^m = 1}$.
Therefore \eqref{V_1_violation} leads to a contradiction.

\medskip

The argument for $V_2$ is analogous, constructing $\phi$ that is below $p^+$ in a neighborhood of
$(x_2, t_2)$.
\end{proof}

Hence $t_0 < T$ leads to a contradiction.
Indeed,
since $|(x_1, t_1) - (x_2, t_2)| \leq 4 r(t_0)$, Lemma~\ref{le:velocity-bounds}, \eqref{weak-grad-order},
Lemma~\ref{le:rho-ext-order} and the Lipschitz continuity of $\vec b$ imply
\begin{align*}
  V_1 \leq \frac{\eta_1}{1 - \rho^-_E(x_1, t_1)} + b(x_1) \cdot \nu \leq \frac{\eta_2}{1 -
    \rho^+_E(x_2, t_2)} + b(x_2) \cdot \nu  + 4 L r(t_0) \leq V_2 + 4L r(t_0).
\end{align*}
But that contradicts \eqref{velocity-gap} as $r'(t_0) = -2 L r(t_0)$.
Therefore $t_0 \geq T$, proving Claim~\ref{cl:no-contact}.

\begin{proof}[Proof of Theorem~\ref{th:almost_cp}]
Since $T > 0$ in Claim~\ref{cl:no-contact} can be taken arbitrary, we conclude that
\begin{align*}
  \set{\rho^- = 1} \subset \set{p^+ > 0}.
\end{align*}
Combining this with Lemma~\ref{le:lim_props}(e), we recover the full set ordering
\begin{align*}
  \set{p^- > 0} \subset \set{\rho^- = 1} \subset \set{p^+ > 0} \subset
  \set{\rho^+ = 1}.
\end{align*}

Since $\interior \set{p^{+,r} = 0} \supset \set{p^+ = 0} \cap \set{t > \tau}$ for any $r_0 > 0$ where
$\tau$ is the solution of $r(\tau)= \tau$,
from Lemma~\ref{le:weak_mon_r} we deduce that $\rho^- \leq \rho^-_E$ in $\set{p^+ = 0} \supset
\set{\rho^+ < 1}$.
Since $\rho^-_E \leq \rho^+_E \leq \rho^+$ on $\set{\rho^+ < 1}$ by
Corollary~\ref{co:transport-comparison}, we conclude that
\begin{align*}
  \rho^- \leq \rho^+.
\end{align*}
Similarly, $p^- = 0$ on $p^+ = 0$. Therefore the comparison principle for the elliptic problem on
$\set{p^+(\cdot, t) >0}$ using Lemma~\ref{le:lim_props}(a) yields
\begin{align*}
  p^-(\cdot, t) \leq p^+(\cdot, t) \qquad t > 0.
\end{align*}
Theorem~\ref{th:almost_cp} follows.
\end{proof}

\subsection{Notion of viscosity solutions of \texorpdfstring{\eqref{hs}}{HS}}
\label{sec:viscosity-solutions}

Before we proceed to the next section, we generalize the almost comparison obtained above to
viscosity solutions. This general comparison principle will be used in
Section~\ref{sec:convergence} to show that a
unique congested zone $\overline{\{\lim_{m\to\infty} \rho=1\}}$ emerges in the limit $m\to \infty$,
and this set satisfies the evolution law given in \eqref{hs}, in the sense of viscosity solutions.

\medskip

We will define the viscosity solutions for \eqref{hs} via barriers in the spirit of \cite{CV}. The
notion via test functions can be developed as in \cite{KP,K03}. In fact, this notion was used in
\cite{AKY} for \eqref{hs} with zero exterior density, as well as in \cite{CKY}. See also
\cite{Pozar14} for the equivalence of these notions in a monotone problem.

We will consider the problem \eqref{hs}, rewritten here using $V = \frac{u_t}{|\nabla u|}$ and $\nu
= -\frac{\nabla u}{|\nabla u|}$ as
\begin{align*}
\left\{
\begin{aligned}
  -\Delta u &= F && \text{in } \set{u > 0},\\
  u_t &= \frac{\abs{\nabla u}^2}{(1 - \rho^E)_+} - \vecb \cdot \nabla u  && \text{on } \partial\set{u > 0},
\end{aligned}
\right.
\end{align*}
where $F \in C(\Rn)$, $\inf F > 0$, and $\rho^E \in C(\Rn)$, $\rho^E \geq 0$ and in this paper
$\rho^E$ solves the transport
equation \eqref{te} with a given Lipschitz vector field $\vecb \in C(\Rn, \Rn)$. Recall that $\frac
1{(1 - \rho^E)_+}$ is understood as $+\infty$ when $\rho^E
\geq 1$.

\begin{remark}
  In general, it is possible to consider nonlinear uniformly elliptic equations like $\mathcal F(D^2u)
  = F$, and $\rho^E$ can be any continuous function.
\end{remark}

We use the notion of the strict separation and a parabolic neighborhood as defined in
\cite{Pozar14,KP}.

\begin{definition}[Parabolic neighborhood and boundary]
\label{def:parabolic-nbd}
\ \\A nonempty set $E \subset \Rn\times \R$ is called a \emph{parabolic neighborhood} if $E = U \cap \set{t \leq \tau}$ for some open set $U \subset \Rn\times \R$ and some $\tau \in \R$.
We say that $E$ is a parabolic neighborhood of $(x,t) \in \Rn\times\R$
if $(x,t) \in E$.
Let us define $\partial_P E := \cl{E} \setminus E$, the \emph{parabolic boundary} of $E$.
\end{definition}

\begin{definition}[Strict separation]
\label{def:strict-separation}
Let $E \subset \Rn \times \R$ be a parabolic neighborhood,
and $u, v : E \to \R$ be
bounded functions on $E$,
and let $K \subset \cl E$.
We say that $u$ and $v$ are strictly separated
on $K$ with respect to $E$,
and we write $u \strictordwrt{E} v$ in $K$,
if
\[
  u^* < v_* \hbox{ in } K \cap \cl{\set{u > 0}}.
\]
Recall that $u^*$ and $v_*$ are well-defined on $\cl E$.
\end{definition}

We introduce (strict) barriers. Intuitively, these are local strict classical subsolutions or
supersolutions of \eqref{hs}.

\begin{definition}
\label{def:barrier}
Let $U \subset \Rn \times \R$ be a nonempty open set and let $\phi \in C^{2,1}(\cl U)$.
We say that $\phi$ is a \emph{subbarrier} of \eqref{hs} on $U$ if $\phi$ satisfies on $\cl U$
\begin{enumerate}[(i)]
  \item $-\Delta \phi < F$ on $\cl{\set{\phi > 0}}$,
  \item $\phi_t < \frac{\abs{\nabla\phi}^2}{(1-\rho^E)_+} - \vecb \cdot \nabla\phi$ when $\phi = 0$,
  \item $|\nabla\phi| > 0$ when $\phi = 0$.
\end{enumerate}
A \emph{superbarrier} is defined analogously by reversing the inequalities in (i)--(ii),
and requiring additionally that $\rho^E < 1$ when $\phi \leq 0$.
\end{definition}

The viscosity solutions are defined via the comparison principle with the barriers as in \cite{KP}.
However, since we are dealing with a nonmonotone problem, we explicitly add the evolving set into
the definition of a viscosity subsolution, similarly to \cite{CKY}. For a related developments see
\cite{Barles-Souganidis,Cardaliaguet-Rouy}.

\begin{definition}
\label{def:visc-sol}
Suppose that $\mathcal{N} \subset \Rn \times \R$.
We say that a pair $(u, \Sigma)$ of a non-negative upper semi-continuous
function $u: \mathcal{N} \to [0,\infty)$ and a closed set $\Sigma \subset \Rn \times \R$ is
a \emph{viscosity subsolution} of \eqref{hs} on $\mathcal{N}$  if
$\set{u > 0} \subset \Sigma$ and
for every bounded parabolic neighborhood $E \subset \mathcal{N}$, $E = U \cap \set{t \leq \tau}$ for some
open set $U$ and $\tau \in \R$,
and every superbarrier $\phi$ on $U$
such that $u \strictordwrt{E} \phi$ on $\partial_P E$ and $\Sigma \cap \partial_P E \subset
\set{\phi > 0}$,
we also have $u \strictordwrt{E} \phi$ on $\cl E$ and $\Sigma \cap \cl E \subset \set{\phi > 0}$.

\medskip

Similarly, a non-negative
lower semi-continuous function $u: \mathcal{N} \to [0,\infty)$ is
a \emph{viscosity supersolution} of \eqref{hs} if
$\set{\rho^E \geq 1} \cap \mathcal{N} \subset \cl{\set{u > 0}}$,
and for every bounded parabolic neighborhood $E \subset \mathcal{N}$
and every subbarrier $\phi$ on $U$ such that $\phi \strictordwrt{E} u$ on $\partial_P E$, we also have $\phi \strictordwrt{E} u$ on $\cl E$.

\medskip

Finally, $u$ is a \emph{viscosity solution} if $u_*$ is a viscosity supersolution and $(u^*,
\cl{\set{u_* > 0}})$ is a viscosity subsolution.  We say $u$ is a \emph{viscosity solution of
\eqref{hs} in $\mathcal{N} = \Rn \times (0,T)$, $T > 0$, with initial density $\rho^0$} if it is a
viscosity solution in $\mathcal N$, $\rho^0$
is of the form $\rho^0 = \max(\chi_{\Omega^0}, \rho^{E,0})$ as in \eqref{initial}, $\rho^E$ is the
solution of \eqref{te} with initial data $\rho^{E,0}$, and the initial data for $u$ is given as
$\cl{\{u_* > 0\}}_0 = \cl{\{u^* > 0\}}_0 = \Omega^0$.
\end{definition}

\subsubsection{Comparison principle}

Following the proof of the almost comparison theorem, Theorem~\ref{th:almost_cp}, we can establish
a comparison principle for strictly ordered solutions.

\begin{theorem}
  \label{th:comparison_strict}
  Let $\mathcal{N} \subset \Rn \times \R$ be a bounded parabolic neighborhood and $\vecb \in Lip(\cl{
  \mathcal{N}}, \Rn)$. Let $\rho^E_u$, $\rho^E_v$ be two continuous functions such that there is $\delta > 0$ with
  $\rho^E_u(x_1, t_1) \leq \rho^E_v(x_2, t_2)$ for all $|(x_1, t_1) - (x_2, t_2)| < \delta$, and let
  $F_u$, $F_v$ be bounded uniformly continuous functions with $0 < F_u < F_v - \delta$.

  Consider $(u, \Sigma)$ and $v$ that are respectively a
  viscosity subsolution and viscosity supersolution of \eqref{hs} in the domain $\mathcal{N}$ with $F=F_u$, $\rho^E=\rho^E_u$ and
  $F=F_v$, $\rho^E=\rho^E_v$. Then the following holds:
\[
 \hbox{ If }u \prec_\mathcal{N} v \hbox{ on } \partial_P \mathcal{N} \hbox{ and }\Sigma \cap
  \partial_P \mathcal{N} \subset \set{v > 0}, \hbox{ then } u \prec_{\mathcal{N}} v \hbox{ on }\cl{\mathcal{N}}.
  \]
\end{theorem}

\begin{remark}
  \label{re:comparison-strict-our-data}
  By Lemma~\ref{le:rho-ext-order}, $\rho_u^E$ and $\rho_v^E$ that are solutions of the transport
  equation \eqref{te} with initial data $\rho_u^{E,0} \prec \rho_v^{E,0}$ satisfy the assumptions
  of the above theorem for any $\mathcal{N} = \Rn \times (0, T)$.
\end{remark}

\begin{proof}
  The proof in fact follows closely the proof of the almost comparison
  Theorem~\ref{th:almost_cp}, and we will therefore only give a brief outline.

  Due to the strict ordering, we can assume that $u$ and $v$ are defined on $\cl{\mathcal{N}}$. By
  taking the sup-convolution of $u$ and $\Sigma$ and the inf-convolution of $v$ over a sufficiently
  small decreasing-in-time set as in
  the proof of the almost comparison, we may assume that $u$, $v$, and $\Sigma$ have interior,
  exterior ball properties, etc., and $u$ and $v$ are still viscosity solutions of \eqref{hs} with
  ordered (by the assumptions) sup-convolutions of $F_u$, $\rho_u^E$, and inf-convolutions of
  $F_v$, $\rho_v^E$, respectively.

  As in the proof of the almost comparison principle, we define the contact time $\hat t$ as the
  supremum of the times $s$ so that the comparison principle holds on $\mathcal{N} \cap \set{t \leq s}$.
  If the comparison does not hold for $\mathcal{N}$, we have $\hat t < \infty$.

  We must have $\Sigma \cap \cl{\mathcal{N}} \cap \set{t \leq \hat t} \cap \set{v = 0} \neq \emptyset$.
  Indeed, if not, since both $\Sigma$ and $\set{v = 0}$ are closed, the above intersection is empty and
therefore it is empty even if we replace $\hat t$ with some $s > \hat t$. That means that $u$ must
cross $v$ in $\set{v > 0}$, but this is a contradiction with the elliptic comparison principle.

Finally, the contact points $(\hat x, \hat t) \in \Sigma \cap \set{v = 0}$ all lie on the
boundary of $\Sigma_{\hat t}$ and $\partial \set{v(\cdot, \hat t) > 0}$ since the sets cannot
expand discontinuously (into $\set{\rho_u^E < 1}$ for $\Sigma$) by a barrier argument.

Therefore we are in the same setting as in the proof of Claim~\ref{cl:no-contact}, therefore we can
construct barriers for $(u, \Sigma)$ and $v$ to reach a contradiction.
\end{proof}

\section{Convergence of \texorpdfstring{$\rho_m$ and $p_m$ as $m\to\infty$}{rhom and pm as m to
infty}}
\label{sec:convergence}

In this section we discuss the convergence of density and pressure variables as $m\to\infty$. First we combine Theorem~\ref{th:almost_cp} and the $L^1$ contraction property,
Lemma~\ref{le:contraction}, between solutions of \eqref{pme} to deduce uniform convergence of the
density $\rho_m$ as $m\to\infty$.  The main step to do so is to show that the {\it congested zone},
where $\rho_m$ uniformly converges to $1$ given a convergent subsequence, is independent of the
choice of the subsequence. We show this by taking the upper and lower limit of $\rho_m$ and show
that their congested zones coincide. We then characterize the congested zone with the free boundary
problem~\eqref{hs} (Corollary~\ref{density}).

\medskip

For Sections~\ref{sec:density-convergence}--\ref{sec:characterization}, we assume that $\rho^0$ satisfies \eqref{initial} and  $F$ is strictly
positive, so that all results from Section~\ref{sec:almost-cp} apply.

\subsection{Density convergence}
\label{sec:density-convergence}

\medskip
Let $\rho_m$ and $p_m$ denote the density and pressure solutions of \eqref{pme} with initial data $\rho_m^0$ satisfying \eqref{initial:convergence}.
Following the notions from Section~\ref{sec:preliminaries}, let us define the lower and upper
limits of the density and pressure variables as
\begin{equation}\label{half_limits_0}
   \bar{\rho}:=\halflimsup_{m\to\infty} \rho_m, \quad \ubar{\rho}:= \halfliminf_{m\to\infty} \rho_m  \,
\end{equation}
and
\begin{equation}
\bar{p}: = \halflimsup_{m\to\infty} p_m, \quad \ubar{p}:= \halfliminf_{m\to\infty} p_m.
\end{equation}

To apply Theorem~\ref{th:almost_cp}, next we consider a decreasing sequence of initial
data $\rho^+_{0,k}$ such that it is strictly larger than $\rho^0$ in the sense of \eqref{initial-data-order} and  it converges to $\rho^0$
from above.

To construct such sequence, recall that  we denote $\{\rho^0=1\}$ by $\Omega^0$. Using this notion one can define
\[\Omega^+_{0,k}:= \{y: d(x,\Omega^0) \leq \tfrac{1}{k}\}, \quad c_k:=
  \sup_{\R^n\setminus\Omega^+_{0,k}} \rho^0, \qquad k \in \mathbb{N},
\]
and define $\rho^+_{0,k}(x,t)$ to be $1$ on $\Omega^+_{0,k}$ and  to be $\rho^0+
\frac{1-c_k}{k}\rho^0$ on $\R^n \setminus \Omega^+_{0,k}$. Note that  $\rho^+_{0,k}$ satisfies
\eqref{initial} since $\rho^0 + \frac{1-c_k}{k} \leq c_k+ \frac{1-c_k}{k} <1 $ on $\R^n\setminus\Omega^+_{0,k}$.
\medskip

Let us denote the corresponding solutions $\rho_m$ of \eqref{pme} with a larger source term
$f^+_k:= f+\frac 1k$ by $\rho^+_{m,k}$ and introduce its lower limit
$\rho^+_k := \halfliminf_{m\to\infty} \rho^+_{m,k}$. Similarly to the above construction, we can consider an increasing sequence
of initial data $\rho_{0,k}^-$ which is strictly smaller than $\rho^0$ with support $\Omega^-_{0,k}$ and it converges to $\rho^0$
from below. Let us denote the corresponding $\rho_m$ solving \eqref{pme} with a smaller source term $f^-_k:= f-\frac{1}{k}$ by $\rho^-_{m,k}$ and its upper limit by
$\rho^-_k := \halflimsup_{m\to \infty} \rho^-_{m,k}$.   The corresponding pressure functions will be denoted similarly as
$p^{\pm}_{m,k}$ and $p^\pm_k$.
\medskip

The aforementioned approximating sequences have several useful ordering properties. Note that from the comparison principle for \eqref{pme} it follows that $\{\rho^+_{m,k}\}_k$ is monotone decreasing and $\{\rho^-_{m,k}\}_k$ monotone increasing, and $\rho^+_{m,k} > \rho^-_{m,j}$ for any $j,k$. Furthermore their half-relaxed limits are ordered with respect to $k$:
\begin{equation}\label{order}
\rho^-_k  \leq \ubar{\rho} \leq \bar{\rho} \leq \rho^+_k \quad\hbox{ for any } k \in
\mathbb{N}.
\end{equation}
where the first and third inequalities are due to Theorem~\ref{th:almost_cp}, and the second is by definition of $\bar{\rho}$ and $\ubar{\rho}$.
\medskip

Let us first show that $\bar{\rho}$ agrees with $\ubar{\rho}$ almost everywhere,  using \eqref{order} and the $L^1$ contraction.

\begin{lemma}\label{approximation}
For any given $t>0$,
$\rho^-_k(\cdot,t), \rho^+_k (\cdot,t)$ converge in $L^1(\R^n)$ to  $\ubar{\rho}(\cdot,t)$ as $k \to \infty$. Furthermore $\bar{\rho}(\cdot,t) =\ubar{\rho}(\cdot,t)$ a.e.
\end{lemma}
\begin{proof}
To prove the convergence, first observe that
\begin{align*}
\int (\rho^+_k(\cdot, t) - \rho^-_k(\cdot, t))dx &\leq \liminf_{m\to\infty} \int
(\rho^+_{m,k} - \rho^-_{m,k})(\cdot,t)  dx\\
&= \liminf_{m\to\infty} \int |\rho^+_{m,k} - \rho^-_{m,k}| (\cdot,t)dx\\
&\leq e^{t \norm{f}_\infty} \liminf_{m\to \infty}  (\|\rho^+_{0,k} - \rho^-_{0,k}\|_{L^1(\R^n)} + \frac{1}{k} \|\rho^0\|_{L^1(\R^n)})
\end{align*}
where we have used Fatou's lemma and the fact that  $\rho^+_{m,k} - \rho^-_{m,k} \geq 0$ for the first inequality, and Lemma~\ref{le:contraction}
 for the last inequality.  Note that the last term on the right converges to zero as $k\to\infty$ by construction. Thus we have, again from Fatou's lemma,
\begin{equation*}
  \int \lim_{k\to\infty}(\rho^+_k(\cdot, t) - \rho^-_k(\cdot, t)) dx = 0.
\end{equation*}

Now from \eqref{order} we conclude that $\bar{\rho}(\cdot,t) = \ubar{\rho}(\cdot,t) = \lim \rho^+_k(\cdot, t)$ almost everywhere.

\end{proof}

\begin{corollary}\label{cor:contraction}
  The $L^1$ contraction holds for the limit density $\rho(\cdot,t) := \bar{\rho}(\cdot,t) = \ubar{\rho}(\cdot,t)$, defined almost everywhere. More precisely if $\rho$ and $\tilde{\rho}$ corresponds to two different limit densities, then

 \begin{equation}\label{infinity_contraction}
  \|\rho(\cdot,t) - \tilde{\rho}(\cdot,t) \|_{L^1 (\R^n)} \leq e^{t
  \norm{f}_\infty}\|\rho(\cdot,0) - \tilde{\rho}(\cdot,0)\|_{L^1(\R^n)}.
  \end{equation}

\end{corollary}



In the next section we show that the congested zone is the same for both half-limits $\bar{\rho}$ and
$\ubar{\rho}$, therefore characterizing the set as the unique congested zone generated by the densities
$\rho_m$ as $m\to\infty$.  Secondly we show that the congested zone can be uniquely obtained by solving the
limiting pressure problem \eqref{hs}. For this purpose it will be useful to consider the following characterization of the pressure half-limits $\bar{p}$ and $\ubar{p}$.

\medskip

\begin{lemma}\label{viscosity_limit}
  $(\bar{p}, \set{\bar{\rho} = 1})$ is a viscosity subsolution of \eqref{hs} while $\ubar{p}
$ is a viscosity supersolution of \eqref{hs}.
\end{lemma}

\begin{proof}
  Let us show the subsolution part.
  We write $\Omega := \set{\bar{\rho} = 1}$. This is a closed set  and
  $\set{\bar{p} > 0} \subset \Omega$ by
  Lemma~\ref{le:lim_props}.
  Note that $\bar{\rho} \leq \rho^E$ outside of $\Omega$ due to Lemma~\ref{approximation} and
  Lemma~\ref{le:weak_mon_r}.

  Let $\phi$ be a superbarrier for \eqref{hs} on an open set $U
\subset \Rn \times \R$ and let $E = U \cap \set{t \leq \tau}$ be a parabolic neighborhood with $E
\subset Q$ such that $\bar{p} \prec_E \phi$ on $\partial_P E$ and $\Omega \cap \partial_P E \subset
\set{\phi > 0}$.
We need to prove $\bar{p} \prec_E \phi$ on $\cl E$ and $\Omega \cap \cl E \subset \set{\phi > 0}$.

Suppose that this is not true. Introduce the contact time $\hat t$ defined as the supremum of times
$s$ for which the above holds for parabolic neighborhood $E \cap \set{t \leq s}$. Since it does not hold
for $E$ itself, we have $\hat t < \infty$.

We must have $\Omega \cap \set{\phi \leq 0} \cap \cl E \cap \set{t \leq \hat t} \neq \emptyset$.
Indeed, since both $\Omega$ and $\set{\phi \leq 0}$ are closed, if the above intersection is empty,
it is empty even if we replace $\hat t$ with some $s > \hat t$. That means that $\bar{p}$ must
cross $\phi$ in $\set{\phi > 0}$, but this is a contradiction with the elliptic maximum principle
by Lemma~\ref{le:lim_props}(a).

Finally, the contact points $(\hat x, \hat t) \in \Omega \cap \set{\phi \leq 0} \cap \cl E \cap \set{t \leq \hat t}$ all lie on the
free boundary $\set{\phi = 0}$ of $\phi$ since $\rho^E < 1$ on $\set{\phi = 0}$ and therefore
Lemma~\ref{le:left-accessible} applies.
By perturbing $\phi$ in the standard way, i.e., adding $\e|x - \hat x|^2 + \e |t - \hat t|^2$ if
necessary, we may assume that $(\hat x, \hat t)$ is the only contact
point.

By Proposition~\ref{pr:barrier-sequence} there exist supersolutions $\varphi_\rho^m$, $\varphi_p^m$
of \eqref{pme} that
approximate $\phi$ from above in a parabolic neighborhood $\mathcal N$ of $(\hat x, \hat t)$. By comparison
principle, $p_m \leq \varphi_p^m$ and $\rho_m \leq \varphi_\rho^m$ for $m$ large enough so that the
boundary data on the parabolic boundary of $\mathcal N$ are strictly ordered. Since this ordering
holds even for small translations of the barriers (uniformly in $m$), we deduce that $\bar{\rho}(\hat
x, \hat t) < 1$, a contradiction with $(\hat x, \hat t) \in \Omega$.

\medskip

The proof for $\ubar{p}$ is analogous, but we need to additionally show that $\set{\rho^E \geq 1} \subset
\cl{\set{\ubar{p} > 0}}$. But this follows from Lemma~\ref{le:liminf-interior-dense}(c). Furthermore, to
show that the contact point is on the free boundary, we use Lemma~\ref{le:char} to show that
$\set{\ubar{p} = 0}$ cannot expand discontinuously in time.
\end{proof}

\subsection{ Characterization of the congested zone and uniform convergence of the density variable }
\label{sec:characterization}

\medskip

Our first goal is to show that the congested zones and the pressure supports from each half-limits
all coincide.

\begin{proposition}
  \label{pr:congested-zone}
  The congested zone is well-defined, that is,
\begin{equation}\label{support}
\Omega := \qquad \{\bar{\rho}=1\} = \overline{\{\ubar{\rho}=1\}} = \overline{\{\ubar{p}>0\}}= \overline{\{\bar{p}>0\}}.
\end{equation}
\end{proposition}

The proof of the above proposition is split into Lemma~\ref{p_1} and Lemma~\ref{le:congested-2}.
Note that we already know that
\begin{align}
  \label{known-inclusions}
  \overline{\{\ubar{p} > 0\}} = \overline{\{\ubar{\rho}  =1\}} \subset \{\bar{\rho} = 1\}, \quad \text{ and }
  \quad
  \overline{\{\bar{p} > 0\}} \subset \{\bar{\rho} = 1\}.
\end{align}
This follows respectively from Lemma~\ref{le:liminf-interior-dense}(d),
$\ubar{\rho} \leq \bar{\rho}$ and the fact that $\{\bar{\rho} = 1\}$ is closed by Lemma~\ref{le:lim_props}(d),
 and finally Lemma~\ref{le:lim_props}(e).
We now show the last equality in \eqref{support}.

 \begin{lemma}\label{p_1}
   $\bar{p}=0$ outside of $\overline{\{\ubar{p}>0\}}$. In particular, $\overline{\{\bar{p}>0\}} = \overline{\{\ubar{p}>0\}}$.
   \end{lemma}

   \begin{proof}
     To see this first observe that by \eqref{known-inclusions} for each $t>0$ we have
     $\{\ubar{\rho}(\cdot,t)=1\} \subset \overline{\{\ubar{p}>0\}}_t$ and
     $\overline{\{\bar{p}(\cdot,t)>0\}}\subset\{\bar{\rho}(\cdot, t)=1\}$.
    Since $\bar{\rho}(\cdot,t)=\ubar{\rho}(\cdot,t)$ a.e. due to Lemma~\ref{approximation}, it follows that
    \begin{align}
      \label{diff-measure}
      |\overline{\{\bar{p}(\cdot,t)>0\}} \setminus \overline{\{\ubar{p}>0\}}_t| = 0.
    \end{align}
Since $\bar{p}(\cdot,t)$ satisfies $-\Delta \bar{p} (\cdot,t) \leq F=f - \divo \vec{b}$ by
Lemma~\ref{le:lim_props}(a),  we can conclude that $\bar{p}(\cdot,t)=0$ outside of $\overline{\{\ubar{p}>0\}}_t$.

More precisely, let us write $u = \bar{p}(\cdot, t)$ and
set $M = \sup F$. Since $u$ is a viscosity solution of $-\Delta u \leq F$, for $x_0 \in \Rn$, $u(x) + \frac M{2n} |x - x_0|^2$
is subharmonic and by the mean value property
\begin{align}
  \label{sub_mean}
  u(x_0) \leq |B_r(x_0)|^{-1}\int_{B_r(x_0)}  u(y) \;dy + C_n M r^2, \quad x_0 \in \Rn,\ r>0.
\end{align}
Note that the mean value property applies even to merely upper semi-continuous functions by a
monotone approximation by continuous functions.
If $B_r(x_0) \cap
\overline{\{\ubar{p} > 0\}}_t = \emptyset$ for some $x_0$, $r$, then $|B_r(x_0) \cap
\overline{\{\bar{p}(\cdot, t) > 0\}}| = 0$ by \eqref{diff-measure}. Therefore $\bar{p}(x_0, t) \leq C_n M r^2$ by \eqref{sub_mean}. We conclude that $\bar{p} = 0$ outside of $\overline{\{\ubar{p} > 0\}}$.
 \end{proof}

Now it remains to prove the first equality in \eqref{support}, we show this by proving the following:
\begin{lemma}
  \label{le:congested-2}
\[
\{\bar{\rho}=1\} = \overline{\{\ubar{p}>0\}}.
\]
In particular \eqref{support} holds.
\end{lemma}

Note that this is easy to prove if the densities $\rho_m$ strictly increase in time, since in that
case $\bar{\rho}(\cdot,t) \leq \ubar{\rho}(\cdot,t+\e)$ for any $\e>0$, which yields $\{\bar{\rho}=1\} =
\overline{\{\ubar{\rho}=1\}} =\overline{\{\ubar{p}>0\}}$. However, such monotonicity is not true for our case, due to the presence of the drift. Still the main idea in the proof is instead to rely on the monotonicity property of the lower limit pressure $\ubar{p}$ along the streamlines (Lemma~\ref{le:char}).
\begin{proof}

Let $(x_0,h)$ be a point outside of $\overline{\{\ubar{p}>0\}}$, and let $X(t)$ be the corresponding
characteristic path with $X(h)=x_0$.  We claim that $x_0$ lies outside of $\{\bar{\rho}(\cdot,h)=1\}$.
This claim, which yields $\{\bar{\rho}=1\} \subset \overline{\{\ubar{p}>0\}}$, is sufficient to conclude
since we know $\overline{\{\ubar{p}>0\}} \subset \overline{\{\ubar{\rho}=1\}} \subset \{\bar{\rho}=1\}$ by
\eqref{known-inclusions}.

To prove the claim, let $L$ denote the Lipschitz constant for the vector field $\vec{b}$. We will show above Lemma for $h \leq \frac{1}{4L}$, which is enough to conclude for general $h>0$ by iterating the argument below.

\medskip

First note that, since $(X(h),h)$ lies outside of $\overline{\{\ubar{p}>0\}}$, the characteristic
path before time $h$, Lemma~\ref{le:char} yields that the characteristic path before time $h$,
$\mathcal{P}:=\{(X(\tau), \tau) : 0\leq \tau\leq h\}$, lies outside of
$\overline{\{\ubar{p}>0\}}$. Moreover $\Omega^0 \subset \overline{\{\ubar{p}>0\}}_0$ by Lemma~\ref{le:up-positive-init}. Hence Lemma~\ref{le:order-initial}
  yields

\begin{equation}\label{order1}
\{\bar{\rho}(\cdot,0)=1\} \cap B_{2r}(X(0)) = \emptyset
\end{equation}
for some $r>0$, and
\begin{equation}\label{zero}
\bar{p}=0 \hbox{ in } \quad \mathcal{N}:= \bigcup_{0\leq \tau\leq h} (B_{2r}(X(\tau))\times\{\tau\})
\subset (\overline{\{\ubar{p}>0\}})^\compl.
\end{equation}

\medskip

  We will now show that the set $\{\bar{\rho}(\cdot,h)=1\}$ stays out of $B_{r/2}(X(h))$.  To do so
  we will invoke the subsolution property of $(\bar{p}, \{\bar{\rho}=1\})$ shown in
  Lemma~\ref{viscosity_limit} to carry out a barrier argument. More precisely, we will construct a
  superbarrier $\phi$ of \eqref{hs} below, based on \eqref{order1}--\eqref{zero} above.

\medskip

Consider a radial function $w: B_r(0)\to \R$ such that $w=0$ in $|x|\leq r-\delta t$, $w = \e$ at $|x|=2r$ and
  $-\Delta w = \sup_{\mathcal{N}} F $ for $r-\delta t\leq |x|\leq 2r$.  Note that $|\nabla w| = O(\e)$ for $|x|\leq 2r$.  Next let us define $\delta:= \frac{r}{2h}> 2Lr$ and consider the function
 \[
 \phi(x,t):=\psi(x,t):=w(x-X(t),t)
 \]
 in the domain $\mathcal{N}$. Observe that the support of $\phi$ moves with the normal velocity
  \[
V = X'(t) \cdot\nu  +\delta = \vec{b}(X(t),t)\cdot\nu + \delta \geq \vec{b}(x,t)\cdot \nu  - Lr + \delta \quad \hbox { for } (x,t)\in\partial\{\phi>0\},
  \]
  where the inequality holds due to the fact that $\partial\{\phi(\cdot,t)>0\} \subset B_r(X(t))$
  for $0\leq t\leq h$ and the Lipschitz continuity of $\vec{b}$ with respect to the space variable.
  Since $\delta >2Lr$ and $|\nabla \phi| = O(\e)$,  if $\e \ll \delta$, we conclude that
  \[
  V \geq \vec{b}(x,t) \cdot \nu + \tfrac \delta2 \geq \vec{b}\cdot\nu + |\nabla\phi| \quad \hbox{ on } \partial\{\phi>0\}.
  \]

 Hence  with a sufficiently small choice of $\e$  it follows that  $\phi$ is a superbarrier for
\eqref{hs} in the domain $\mathcal{N}$.

\medskip

Since $\bar{p}=0$ in $\mathcal{N}$ from \eqref{zero}, $\bar{p} \leq \phi$ in $\mathcal{N}$. Furthermore $\{\bar{\rho}(\cdot,0)=1\}$ is outside of $\{\phi(\cdot,0)=0\} = B_r(X(0))$ due to \eqref{order1}. Thus by the subsolution property of $(\bar{p}, \{\bar{\rho}=1\})$ for \eqref{hs}, the set $\{\bar{\rho} = 1\}$ is contained in the support of $\phi$ in $\mathcal{N}$. We can conclude now since
\[
\{\bar{\rho}(\cdot,h)=1\}\subset \{\phi(\cdot,h)=0\} = B_{r/2}(X(h)).
\]
\end{proof}







With \eqref{support} proven and the congested zone $\Omega$ well-defined, we proceed to
characterizing the set $\Omega$ using the free boundary problem \eqref{hs}. Below we show that the
smallest supersolution of \eqref{hs} with ``initial data $\rho^0$'' has the same support as the one
given by the limit solutions $\bar{p}$ and $\ubar{p}$.

\medskip

\begin{lemma}\label{visc:minimal}
Consider the function $U: Q \to \R$ defined as
\begin{equation}\label{minimal}
U:= \inf \{ p: p \hbox{ is a supersolution of }\eqref{hs} \hbox{ with external density }\rho^E
\hbox{ and }  p(\cdot,0) > 0 \text{ in }\interior \Omega^0\}.
\end{equation}
Then $U$ is a viscosity solution of \eqref{hs} with initial data $\rho^0$, and $\overline{\{U_*>0\} }=\overline{\{U^*>0\}}= \Omega$.
\end{lemma}

\begin{proof}
  Recall that we understand $U_*$ as the largest LSC function on $\cl Q$ that is smaller than $U$
  on $Q := \Rn \times (0,\infty)$. $U^*$ is understood similarly. This technicality is necessary
  since $U$ as defined in \eqref{visc:minimal} is clearly $0$ at $t = 0$.
  Note that the set of eligible supersolutions in \eqref{visc:minimal} is nonempty since $\ubar{p}$
  belongs to the set by Lemma~\ref{le:up-positive-init}.

 First we claim that $U_*$ is a viscosity supersolution and $(U^*, \overline{\{U^*>0\}})$ is a
subsolution of \eqref{hs} in the sense of Definition~\ref{def:visc-sol}. This part of the proof is parallel to the standard Perron's method in viscosity solutions theory, and thus we refer to \cite{CIL} and \cite{K03}.

Next we check the initial data, i.e., that
\begin{equation}\label{claim000}
\overline{\{U_*>0\}}_0 = \overline{\{U^*>0\}}_0 =  \Omega^0.
\end{equation}
  Let us mention that the proof of Lemma~\ref{le:order-initial} yields that
\begin{equation}\label{matching}
\bar{\rho}(\cdot,0) = \rho^0, \quad \ubar{\rho} (\cdot,0) = (\rho^0)_*.
\end{equation}

 Since $U \leq \ubar{p}$, we deduce $U^* \leq \bar{p}$ and thus
\[
\overline{\{U^*>0\}}_0 \subset  \overline{\{\bar{p}>0\}}_0 = \{\bar{\rho} = 1\}_0 = \Omega^0,
\]
where the last equality is due to \eqref{matching}.

\medskip

The other inequality follows from a simple barrier argument.
Let us set
\[\phi(x,t) := \mu(t)\pth{1 -\frac{|x - X(t, x_0)|^2}{\delta^2} e^{2Lt}},\]
where $L > 0$ is a
Lipschitz constant for $\vecb$.
For given $\delta > 0$ there exists $\tilde \mu = \tilde\mu(\delta) > 0$ such that $\phi$ is as subbarrier for
\eqref{hs} whenever
$x_0 \in \Rn$, $\mu \in C^1([0,\delta])$, $\mu \leq \tilde \mu$.
If $p$ is an admissible supersolution in \eqref{visc:minimal} and $x_0 \in \interior \Omega^0$,
we can find $\delta > 0$, $\mu_0 \in (0, \tilde \mu(\delta))$ such that $\phi \leq p$ at $t = 0$ with any $\mu \in
C^1([0,\delta])$ such that $\mu(0) = \mu_0$. By definition of viscosity supersolutions $\phi \leq p$ on $\cl Q$.
Taking a supremum over all such $\mu \leq \tilde \mu$, we deduce that $U_* \geq \phi$ with $\mu = \tilde\mu$.
In particular, $\interior \Omega^0 \subset \{U_* > 0\}$. As $\cl{\interior \Omega^0} = \Omega^0$, we have
\begin{align*}
\Omega^0 \subset \overline{\{U_*>0\}}_0,
\end{align*}
and thus we arrive at \eqref{claim000}.

\medskip


\medskip

It remains to show that

\begin{equation}\label{claim111}
\overline{\{U_*>0\} }=\overline{\{U^*>0\}}= \Omega.
\end{equation}

 Note that $\{U^*>0\} \subset \Omega$ since $U^* \leq \bar{p}$ and $\Omega = \overline{\{\bar{p}>0\}}$. Therefore to show \eqref{claim111} it is enough to show that $\Omega\subset \overline{\{U_*>0\}}$. To show this we claim that
\begin{equation}\label{claim}
\Omega= \mathcal{D}:=\overline{\bigcup_{k>0} \{\rho^-_k=1\}}
\end{equation}

Observe that $\{\rho^-_k=1\}\subset \Omega$ for any $k>0$ from the ordering $\rho^-_k
\leq \bar{\rho}$, hence $\mathcal{D} \subset \Omega$.

 To show that $\Omega \subset \mathcal{D}$, recall that by our assumption $F= f - \divo \vec{b} >0$, the external
density $\rho^E$ strictly increases along streamlines. This yields  $\{\rho^E\geq 1\} =
\overline{\{\rho^E>1\}}$. Also by our construction $\rho^-_{k,E}$, the external density associated
with $\rho^-_{k,0}$, locally uniformly converges to $\rho^E$. Thus it satisfies
\[
\{\rho^E>1\} \subset \bigcup_{k>0}\{\rho^-_{k,E} \geq 1\} \subset \bigcup_{k>0} \{\rho^-_k=1\},
\]
where the second inclusion is due to the fact that $\{\rho^-_{k,E}=1\} \subset
\{\rho^-_k \geq 1\}$ by Lemma~\ref{le:liminf-interior-dense}(c).  As a consequence it follows that
\begin{equation}\label{inclusion000}
\{\rho^E\geq 1\} \subset \mathcal{D}.
\end{equation}

 Now suppose
$(x_0,t_0) \in \{\ubar{p}>0\} \cap \mathcal{D}^\compl$, which then implies $B_r(x_0)\times\{t_0\}
\subset
\{\ubar{p}>0\} \cap \mathcal{D}^\compl$ for some $r>0$. Due to \eqref{inclusion000} this ball lies in the open set
$\{\rho^E<1\}$, and thus we have
$\rho^E< 1- \e$ in $B_r(x_0)\times\{t_0\}$ for some $\e>0$.

By definition of $\mathcal{D}$,  $B_r(x_0)$ lies outside of $\overline{\{\rho^-_k=1\}}$ at
$t=t_0$ for any $k>0$. This and Lemma~\ref{le:char}, as well as the fact that
$\overline{\{p^-_k >0 \} }= \overline{\{\rho^-_k = 1\}}$ by
Proposition~\ref{pr:congested-zone}, it follows that any
streamline $\{(X(t, \cdot), t): t\}$ passing through $B_r(x_0)$ at time $t_0$
lies outside of $\{\rho^-_k=1\}$ for $t\leq t_0$.   Now
Corollary~\ref{co:transport-comparison} states that
\[
  \rho^-_k (\cdot,t_0) \leq \rho^-_{k,E}(\cdot, t_0) \leq \rho^E(\cdot,t_0)<1-\e \hbox{ in }B_r(x_0) \hbox{ for all }k>0.
\]
Since $\rho^-_k(\cdot, t_0)$ converges to $\bar{\rho}(\cdot, t_0)$ a.e. by
Lemma~\ref{approximation}, it follows that
$\bar{\rho}(\cdot,t_0) <1$ a.e. in $B_r(x_0)$, contradicting the fact that $B_r(x_0) \subset
\{\ubar{p}(\cdot, t_0)>0\} \subset \{\bar{\rho}(\cdot,t_0)=1\}$.  Hence such $(x_0,t_0)$ does not exist, which means $\{\ubar{p}>0\}\subset \mathcal{D}$, and thus $\overline{\{\ubar{p}>0\}} = \Omega \subset \mathcal{D}$ and we conclude that \eqref{claim} holds.

\medskip

Now when \eqref{claim} is proved, it remains to show that $\overline{\{U_*>0\}}$ contains
$\mathcal{D}$. This is straightforward because $U_* \geq p^-_k$ from
Theorem~\ref{th:comparison_strict} and Lemma~\ref{viscosity_limit}, and thus $\{\rho^-_k=1\} = \overline{\{p^-_k>0\}} \subset \overline{\{U_*>0\}}$ for any $k>0$.
\end{proof}

The next corollary states that any viscosity solution of \eqref{hs} generates the same pressure support, which is $\Omega$.

\begin{corollary}\label{cor:unique}
Let $u$ be any viscosity solution of \eqref{hs} with initial data $\rho^0$, as defined in Definition~\ref{def:visc-sol}. Then $\overline{\{u_*>0\}}= \Omega$.

\end{corollary}
\begin{proof}
  By definition of $U$ in \eqref{minimal} we have $U\leq u_*$ and thus  $\Omega \subset
  \overline{\{U_*>0\}}\subset \overline{\{u_*>0\}}$. To prove the other inclusion, note that  $u^*
  \leq p^+_{k}$ by Theorem~\ref{th:comparison_strict} and Lemma~\ref{viscosity_limit}. Hence it follows that that
\[
\{u_*>0\} \setminus \Omega \subset \overline{\{p^+_{k}>0\}} \setminus \Omega \hbox{ for any } k>0.
\]

  Also observe that, since $\rho^+_{k}$ decreases to $\bar{\rho}$ and $\int |\rho^+_{k} - \bar{\rho}|(\cdot,t) dx \to 0$, we have
\[
|\overline{\{p^+_{k}>0\}}_t \setminus \Omega_t|  = |\{\rho^+_{k}(\cdot,t)=1\} \cap \{\bar{\rho}(\cdot,t)<1\}|  \to 0 \hbox{ as }k\to\infty.
\]
Hence it follows that  $|\{u_*>0\} \setminus \Omega|=0$. Since this set is open, we conclude $\{u_*>0\} \subset \Omega$, which yields our claim.
\end{proof}

\begin{remark}
Note that the viscosity solution itself may not be unique: this is related to the fact that the inner and outer approximation of harmonic functions with Dirichlet data may be different if the domain is not smooth.

\end{remark}

Now we can show the locally uniform convergence of $\rho_m$ to a limit density given in terms of a viscosity solution $u$, with the price of excluding the boundary of $\{u>0\}$.

\begin{corollary}\label{density}
Let $u$ be a viscosity solution of \eqref{hs} with initial density $\rho^0$, and let $\Omega:= \overline{\{u_*>0\}}$. Then $\rho_m$ with its initial data satisfying \eqref{initial:convergence}
 converges locally uniformly away from $\partial\Omega$ to
\begin{equation}\label{density_limit}
\rho:= \chi_{\Omega} + \rho^E\chi_{\Omega^\compl}.
\end{equation}
\end{corollary}

\begin{proof}
By definition  $\rho_m$ locally uniformly converges to $1$ in  $\{\ubar{\rho}=1\}$. Since $\Omega \subset \{\ubar{p}>0\} \subset \{\ubar{\rho}=1\}$,  we can conclude that $\rho_m$ locally uniformly converges to
$1$ in the interior of $\Omega$.

\medskip

 Outside of $\Omega$, The first statement of Corollary~\ref{co:transport-comparison} yields that
 $\ubar{\rho} \geq \min[1, \rho^E]$.  Also $\bar{\rho}<1$ outside of $\Omega$, and since we now
 know that $\overline{\{\bar{p}>0\}} = \overline{\{\ubar{p}>0\}}$, Lemma~\ref{le:char} guarantees
 that the assumption of the second statement in Corollary~\ref{co:transport-comparison} to hold at
 any point in $\Omega^\compl$. Therefore we have $\bar{\rho}\leq \rho^E$ in $\Omega^\compl$. Hence  we conclude that
\[
\ubar{\rho} = \bar{\rho}=\rho^E\hbox{ in } \Omega^\compl,
\]
which yields the local uniform convergence of $\rho_m$ to $\rho^E$ in $\Omega^\compl$.
\end{proof}

To describe the pressure variable convergence,  let us define
\[
p^{\rm in} := \sup\{ h \in C^{\infty}(Q): -\Delta h (\cdot,t) \leq F \hbox{ in }\Omega, \quad h \leq 0 \hbox{ in }  Q\setminus {\rm Int}(\Omega)\},
\]
and
\[
p^{\rm out}:= \inf\{h\in C^{\infty}(Q): -\Delta h(\cdot,t) \geq F \hbox{ in } \Omega ,\quad  h \geq 0\}.
\]
It is well-known that $p^{\rm in}$ coincides with $p^{\rm out}$ when $\Omega$ has sufficiently
regular boundary parts (for instance spatially Lipschitz)  in a local neighborhood. Note that for
$u$ as given in Corollary~\ref{density}, it satisfies $p^{in } \leq u \leq p^{\rm out}$.
\begin{corollary}\label{pressure}
Let $\Omega = \overline{\{u_*>0\}}$ as above, and let $\rho_m^0$ satisfy \eqref{initial:convergence}.
 Then the following holds for $p_m$:
\begin{itemize}
\item[(a)] $p_m$ locally uniformly converges to zero in $Q\setminus \Omega$. \\
\item[(b)] For given  neighborhood $\mathcal{N}$ in $Q$,  suppose $\Omega$ has sufficiently regular
  boundary in $\mathcal{N}$ such that $u = p^{\rm in}=  p^{\rm out}$ in $\Omega \cap\mathcal{N}$. Then
$p_m$ locally uniformly converges to $u$ in $\mathcal{N}$.
\end{itemize}
\end{corollary}

\begin{proof}

(a) follows from the fact that $\Omega = \overline{\{\ubar{p}>0\}}$. To show (b), note that
Lemma~\ref{le:lim_props} (a) yields $\bar{p} \leq p^{\rm out}$ and $\ubar{p} \geq  p^{\rm in}$. Hence
$\bar{p}=\ubar{p}$ in the region $p^{\rm in} = p^{\rm out}$ from which (b) follows.
\end{proof}

\subsection{Further \texorpdfstring{$L^1$}{L1} convergence results}

The following is now a direct corollary of Corollary~\ref{density} and   the $L^1$ contraction \eqref{contraction}:

\begin{corollary}\label{density:gen}
Let $\rho^0 $ be regular, \eqref{F} hold, and let $\rho$ be as given in \eqref{density_limit}. Suppose $\rho_m$ solve \eqref{pme} with its initial data $\rho_m$ converging to $\rho^0$ in $L^1(\R^n)$ as $m\to\infty$. Then
\[
\|\rho_m(\cdot,t) - \rho(\cdot,t)\|_{L^1(\R^n)}  \to 0 \hbox{ as } m\to\infty.
\]
\end{corollary}

\section{Local BV regularity of the congested zone}
\label{sec:bv}

\medskip

Below we show that $\Omega_t =\overline{\{\ubar{\rho}=1\}}_t=\{\bar{\rho}(\cdot,t)=1\}$ is locally a set of finite perimeter, when it does not go through nucleation of new congested region due to the external density increasing to one from below. First we show a perturbation statement which regularizes supersolutions.

\begin{lemma}\label{supersolution}
Let $L$ be the Lipschitz constant of $F=f-\divo\vec{b}$ and $\vec{b}$. Given a supersolution $p$ of
\eqref{hs}, a constant $r>0$ and $r(t):= re^{-Lt}$,
 \[\tilde{p}_r(x,t):= \inf_{|x-y|\leq r(t)} (1+a)p(y, (1+a)t)
 \] is a supersolution $p$ of \eqref{hs} with source term $(1+\tfrac a2)F$ if the constant $a$ satisfies
 \[
 \dfrac{2Lr(t)}{\inf_{\{p>0\}\cap \{t\leq T\}} F} \leq a, \quad 0< a \leq 1.
 \]
 \end{lemma}
 \begin{proof}
Formally this follows from that in the positive set $\tilde{p}_r$ satisfies
\[
-\Delta \tilde{p}_r \geq (1+a)(F(x)-Lr(t)) \geq (1+\tfrac a2)F(x)
\]
by the choice of $a$, and that the set $\{\tilde{p}_r>0\}$ evolves with the normal velocity
\[
V_{x,t} \geq (|\nabla\tilde{p}_r| -\vec{b})(x,t) - Lr(t) - r'(t) \geq |\nabla\tilde{p}_r| -\vec{b}.
\]
For a detailed argument, using the definition of viscosity solutions, we refer to \cite[Lemma~2.5]{K03}.
 \end{proof}

 \begin{proposition}\label{le:perimeter}
   Let $\rho^0 \in L^1(\R^n)$ be continuous, $0\leq \rho^0 = \rho^{E,0} \leq 1$, and assume that there is a constant
$C$ such that for sufficiently small $r>0$
\begin{align}
  \label{conv-assumption}
  \int_\Rn (\rho^r_0 - \rho^0) \;dx \leq Cr, \quad \hbox{ where } \rho^r_0(x):= (1+r) \sup_{B_r(x)} \rho^0.
\end{align}
 Let us consider an open set $\Sigma \subset \R^n$ where we have, for some $t>0$,
\begin{equation}\label{external}
\rho^E<1-\delta \hbox{ in }\Sigma \cap (\Omega_t)^\compl\hbox{ for some }\delta>0.
\end{equation}
Then for this time $t$ the perimeter of our congested zone $\Omega_t$ is bounded in $\Sigma$ with
\begin{equation}\label{perimeter}
\Per(\Omega_t, \Sigma) \leq  C\delta^{-1}e^{(L+\|f\|_{\infty})t},
\end{equation}
where  $L$ is the Lipschitz constant of $\vec{b}$ and $f-\divo \vec{b}$.
\medskip
\end{proposition}

\begin{remark}
  Note that for instance $\rho^0 = \chi_{\Omega^0} + \rho^{E,0}\chi_{(\Omega^0)^\compl}$ with finite
  perimeter set $\Omega^0$ and Lipschitz continuous $\rho^{E,0}$ satisfies \eqref{conv-assumption}.
\end{remark}

\begin{proof}

Below we will prove the statement for the case $\sup \rho^0<1$. When $\rho^0\leq 1$ the same
estimate \eqref{perimeter} holds. To see this, first note that corresponding external density for
initial data $(1-\frac1k)\rho^0$ is $(1-\frac1k) \rho^E$, which, by assumption \eqref{external} and
continuity of $\rho^E$, stays strictly below $1$ in a neighborhood of $\Sigma \subset
(\Omega_t)^\compl$. This fact and  \eqref{infinity_contraction} yields that the corresponding congested
region $\Omega_{k,t}$ associated with initial data $(1-\frac1k)\rho^0$ satisfies  $ |\Omega_t -
\Omega_{k,t}|\to 0$  as $k\to\infty$ at each $t>0$. Now \eqref{perimeter} follows for $\Omega_t$
due to the lower semi-continuity of the perimeter under $L^1$ convergence.

\medskip

Now we proceed with assuming that $\sup\rho^0<1$, and thus $\rho^r_0 \leq 1$ for sufficiently small
$r>0$.  Note that in particular then $\rho^r_0$ satisfies
\begin{equation}\label{order0}
\rho^0(x) <\tilde{\rho}^r(x):=\inf_{|x-y|\leq r} \rho^r_0(y) \hbox{ in } \{\rho^0>0\}.
\end{equation}

Let $p_m$ solve \eqref{pme} with $f$ replaced by $f+r$ and with initial data $\rho^r_0$. Let
$p^r(x,t)$ denote the pressure lower limit, $\liminf_* p_m(x,t)$, and let $\tilde{p}^r_r$ be as
defined in Lemma~\ref{supersolution} where $p=p^r$. By the lemma, $\tilde{p}^r_r$ is a supersolution of \eqref{hs} with the corresponding initial density $\tilde{\rho}^r$. Due to \eqref{order0},  $\tilde{\rho}^r$ is strictly larger than $\rho^0$, and thus by  Theorem~\ref{th:comparison_strict}
\begin{equation}\label{order2}
\bar{p} \leq \tilde{p}^r_r.
\end{equation}

Let $\Omega_{r,t}:= \{\tilde{p}^r(\cdot,t)>0\}$. Then we have the following property:

\medskip

\begin{itemize}
\item[(a)]$\Omega_t\subset\Omega_{r,t}$ for each $t>0$;\\
\item[(b)] $\Omega_{r,t}$ decreases with respect to $r$;\\
\item[(c)]$\Omega_{r/2,t} \subset\{x: d(x,\Omega_{r,t}^\compl) \geq \tfrac r2 e^{-Lt}  \}$;\\
\item[(d)] $|(\Omega_{r,t} - \Omega_t)\cap\Sigma| \leq Ce^{\|f\|_{\infty}t}\delta^{-1} r$. \\
\end{itemize}

(a) and (b) is due to the definition of $\Omega_{r,t}$. (c) is due to the fact that
\begin{equation*}
\rho^{r/2}_0(y) = (1+\tfrac r2)\sup_{|x|\leq \frac{r}{2}} \rho^0(y) \leq \inf_{|x|\leq \frac{r}{2}} \rho^r_0(y)
\end{equation*}
and thus again from Lemma~\ref{supersolution} and Theorem~\ref{th:comparison_strict}
\begin{align*}
p^{r/2}(x,t) \leq \tilde{p}_{r/2}^r(y,t).
\end{align*}

To show (d), note that $\bar{\rho} = \rho^E$ a.e. outside of $\Omega_t$ and

\begin{align*}
\Omega_{r,t} \subset \{p^r(\cdot,t)>0\} \subset\{\rho^r(\cdot,t)=1\},
\end{align*}
where $\rho^r$ is the density lower limit that corresponds to $p^r$. Hence we have
\begin{align*}
\delta |(\Omega_{r,t} \setminus \Omega_t)\cap\Sigma| &\leq \int_{(\Omega_{r,t}\setminus
\Omega_t)\cap \Sigma} (1-\rho^E(\cdot, t))\;dx  \\
& \leq \int (\rho^r -\bar{\rho})(\cdot,t) dx   \leq  e^{\|f\|_{\infty}t}\pth{\int_{\Sigma} (\rho^r_0 -
\rho^0) \;dx + r\|\rho^0\|_{L^1(\R^n)}} \leq Cr e^{\|f\|_{\infty}t},
\end{align*}
where the third inequality is due to Corollary~\ref{cor:contraction} and last inequality comes from \eqref{conv-assumption}.

\medskip

 Now let us define a sequence of sets $\Omega_t^k:= \Omega_{r_k,t}$ with $r_k=2^{-k}$. We claim that for $r\leq e^{-Lt}r_k$ there is at most
$C(t,\delta) r^{1-n}$ balls of radius $r$ covering the boundary of $\Omega_t^k$.
We will only show the claim for $r=e^{-Lt}r_k$. For  $r<e^{-Lt}r_k$ the claim holds due to
\cite[Lemma~2.5]{ACM}, due to the exterior ball property of $\Omega_t^k$ with radius $e^{-Lt}r_k$.

Now let us take an open covering $\mathcal{O}$ of $\partial\Omega_t^{k+1} \cap \Sigma$, consisting of balls of radius $e^{-Lt}r_{k+1}$
with its center on a boundary point. Using Vitali's covering lemma we can take out a family of disjoint balls $\{B_i\}$ in $\mathcal{O}$ such that $\{3B_i\}$ covers the boundary of $\Omega_t^{k+1}$.

\medskip

In each of this disjoint ball $B_i$, $\tilde{B}_i:= B_i\setminus \Omega_t^{k+1}$ takes up at least
one third of the volume of $B_i$, due to the exterior ball property of $\Omega_t^{k+1}$with radius
$e^{-Lt}r_{k+1}$ satisfied at the center of each ball. Also due to $(c)$ at least $1/4$ portion of
$\tilde{B}_i$ is inside $\Omega_t^k$.  Lastly observe that $(a),(b)$ and $(d)$ yield
 \begin{equation}\label{above}
 |(\Omega_t^k \setminus \Omega_t^{k+1})\cap\Sigma| \leq Ce^{\|f\|_{\infty}t}\delta^{-1}r_{k+1}.
 \end{equation}

 From the above observations we conclude that if the total number of the disjoint balls $\{B_i\}$ are $N$, then \eqref{above} yields that
\begin{align*}
 \frac{1}{12}N(e^{-Lt}r_{k+1})^n &\leq \frac{1}{4}\sum_{i=1}^n |\tilde{B}_i| \leq |(\Omega_t^k
 \setminus \Omega_t^{k+1})\cap \Sigma| \leq Ce^{\|f\|_{\infty}t}\delta^{-1} r_{k+1},\\
\intertext{or}
 N &\leq C\delta^{-1}e^{(nL+\|f\|_{\infty})t}(r_{k+1})^{1-n}.
\end{align*}

We have now shown that
\[\mathcal{H}^{n-1} (\partial\Omega^r_t) \leq 4C\delta^{-1}e^{(nL+\|f\|_{\infty})t} \hbox{ in }\Sigma \hbox{ for all } r=2^{-k}.
\]
 Since (d) ensures that $\Omega^r_t$ converges to $\Omega_t$ in measure as $r\to 0$, we can conclude that, from the lower semi-continuity of the perimeter,
\[
\Per(\Omega_t, \Sigma) \leq \liminf_{n\to\infty} \Per(\Omega^r_t, \Sigma) \leq 4C\delta^{-1}e^{(nL+\|f\|_{\infty})t}.
\]
\end{proof}

\subsection{Examples of patch solutions}

\medskip

We finish this section with a discussion of the settings where a {\it patch solution} $\rho =
\chi_{\Omega_t}$ appears. In these cases \eqref{external} is guaranteed and
thus Proposition~\ref{le:perimeter} yields a BV estimate for $\rho$ given that $\Omega^0$ has finite perimeter.  The simplest such case happens when the initial density is a patch.
\begin{lemma}
If $\rho^0$ is a patch then the limit density $\rho$ given in Corollary~\ref{density} is a patch,
i.e. $\rho=\chi_{\Omega_t}$ . If $\Omega^0$ is of finite perimeter then we have
\[
\Per(\Omega_t)\leq  (\Per(\Omega^0)+C|\Omega^0|)e^{(nL+\|f\|_{\infty})t},
\]
where $L$ is a Lipschitz constant of $\vec{b}$ and $F$.
 \end{lemma}
\begin{proof}
The first statement is a direct consequence of Corollary~\ref{density} and the fact that $\rho^E$ stays zero if initially zero.  The second statement follows from Proposition~\ref{le:perimeter} and
\[
 \int (\rho^r_0 - \rho^0)dx \leq (\Per(\Omega^0)+C|\Omega^0|)r.
\]
\end{proof}

We finish with one additional scenario where one can observe  patch solutions after a finite time.

\begin{lemma}
Suppose $f>0$ and $\vec{b} = - \nabla \Phi$ for a $C^3$ potential $\Phi:\R^n\to\R$. Suppose further that $|\nabla \Phi| \neq 0$ except at $x_0$, where $\Phi$ takes its minimum, and $\Phi(x)\to\infty$ as $|x|\to\infty$. If $\rho^0\in L^1(\R^n)$ is positive in a neighborhood of $x_0$, then there exists $T>0$ such that $\rho(\cdot,t) = \chi_{\Omega_t}$ for all $t>T$.

\end{lemma}

\begin{proof}
1. Without loss of generality we may assume that $\min \Phi= 0$.  For a given positive constant $C>0$ and time $T>0$ we construct an expanding domain of the form
\[
\Sigma(t):=(C+\e t - \Phi)_+ \quad \hbox{ for } 0\leq t\leq T,
\]
where $\e>0$ is a small constant to be chosen, and let $h(\cdot,t)$ solve
\[
-\Delta h = \Delta\Phi + f \hbox{  in } \Sigma(t), \quad  h=0 \hbox{ on } \partial\Sigma(t).
 \]
Note that $-\Delta(h-(C+\e t-\Phi)) = f $ in $\Sigma(t)$ with Dirichlet boundary data. Since $|\nabla \Phi|\neq 0$ and $\Phi$ is $C^2$,  the level sets of $\Phi$ are $C^2$-hypersurfaces.   Hence Hopf's lemma applies to $h-(C+\e t -\Phi)$ and we have
\[
(\nabla h + \nabla \Phi) \cdot (-\nu) \geq \delta \hbox{ on }  \partial\Sigma(t) \hbox{ for some } \delta>0,
\]
where $\nu=\nu_{x,t}$ denotes the spatial outward normal vector. This yields that the normal velocity $V=V_{x,t}$ of $\Sigma(t)$ at $x\in\partial\Sigma(t)$ satisfies
\[
V= \dfrac{\e}{|\nabla\Phi|} \leq \delta \leq (-\nabla h -\nabla \Phi) \cdot \nu = |\nabla h| +\vec{b}\cdot\nu,
\]
if $\e=\e(C,T)$ is chosen smaller than $\min[1,\delta \min_{\{C\leq \Phi \leq C+T\}} |\nabla\Phi|]$.
Hence $(h, \Sigma(t))$ is a viscosity subsolution of \eqref{hs} in $\R^n\times [0,T]$ with this choice of $\e(C,T)$, with initial data $\chi_{\Sigma(0)}$ and with external density zero.

\medskip

2. By assumption $\rho^0$ is positive in a neighborhood of $x_0$, and thus for a given $r>0$ there
is a time $t_0>0$ such that $\rho^E(x,t_0) >1$ in a small ball $B_r(x_0)$ (this happens since the
streamlines point toward $x_0$). Then at this time $\{p(\cdot,t)>0\}$ contains $B_r(x_0)$ and thus
$(C-\Phi)_+$ for small enough $C$. Now with this choice of $C=C_1$ and $\e_1 =\e(C_1,T)$ we can
show that $h(x,t+t_0) \leq p(x,t)$ for $t_0\leq t\leq  T$, and repeating this argument with $C_k =
C_{k-1}+\e_{k-1}T$ with $\e_k = \e_k(C_k, kT)$ for further time interval $((k-1)T, kT)$. Since
$\e_k$ do not converge to zero unless $C_k$ tends to infinity, we can show that at some finite time
any sub-level set of $\Phi$ is contained in the pressure support. On the other hand, if $\rho^0$ is
contained in a sub-level set of $\Phi$ then so does $\rho^E$. Putting this together we conclude.
\end{proof}

\appendix

\section{Barriers}

This section is a collection of barriers that are used at various point throughout the paper.
Recall the definition of $F$ in \eqref{F}.

\subsection{Barriers}

We can construct simple radially symmetric ``go with the flow'' barriers that contract exponentially to
account for a possible local compressibility of $\vec b$.
In the formula \eqref{streamline-barrier} below, $\mu$ approximates the solution of \eqref{stream-density-ode} and
$\eta$ is either a ``bump up'' or a ``bump down'' function.

\begin{lemma}[Density barriers]
  \label{le:flow_barrier}
  Let $\e, \delta > 0$ and $x_0 \in \Rn$, $T> 0$. Suppose that $\mu = \mu(t) > 0$ satisfies
  $\mu'(t) \leq (F(X(t, x_0)) - 2\e) \mu(t)$. Let $r > 0$ such that $\abs{F(x)
   - F(X(t, x_0))} < \e$ for all $(x,t) \in N := \set{(x, t): \abs{x - X(t,
  x_0)} \leq r,\ t \in [0, T]}$. Let
  $L$ be the Lipschitz constant of $\vec b$ in $N$.
  Finally, let $\eta \in C(\Rn)$, $\eta \geq 0$, $\supp \eta \subset \cl B_1(0)$, be a radially symmetric function with $\eta
  \in C^2(\cl{\set{\eta > 0}})$, nonincreasing in $\abs{x}$.
  Then there exists $m_0 = m_0(\e,\delta, T, L, r, \eta)$ such that the function
  \begin{align}
    \label{streamline-barrier}
    \psi(x,t) = \mu(t) \eta\pth{\frac{x - X(t, x_0)}{r e^{-Lt}}}
  \end{align}
  is a classical subsolution of \eqref{pme} on  $\set{0 < \psi < 1 - \delta} \cap (\Rn \times (0,T))$ for
  all $m \geq m_0$.
  Note that $\supp \psi(\cdot, t) \subset \cl B_{re^{-Lt}}(X(t, x_0))$.

  If $F \geq 0$ and $\mu(t) = \mu_0 e^{-\delta t}$, the solution of $\mu' = -\delta
  \mu$, the same result holds but with no restriction on $r > 0$.

  Similarly, if instead $\mu'(t) \geq (F(X(t, x_0)) + 2\e) \mu(t)$ and $\eta(x) > 0$ is
  nondecreasing in $\abs{x}$, then
  there exists $m_0 = m_0(\e, \delta, T, M, r, \eta)$ such that the function $\psi$ in
  \eqref{streamline-barrier} is a classical supersolution
  of \eqref{pme} on the same set for all $m \geq m_0$.
\end{lemma}

\begin{proof}
  For $(x,t)$ such that $0 < \psi(x,t) < 1 - \delta$ and $0 \leq t \leq T$ we check \eqref{pme}.

  Let us compute the derivatives when $\psi > 0$. We have
\begin{align}
  \label{psi-derivatives}
  \begin{aligned}
  \psi_t(x,t) &= \mu'(t) \eta(\cdot) + \nabla \psi(x,t) \cdot \pth{-\vec b
    (X(t, x_0)) + L(x - X(t, x_0))},\\
  \nabla \psi(x,t) &= \frac{\mu(t)}{r e^{-Lt}} \nabla \eta(\cdot),\\
  \Delta \psi(x,t) &= \frac{\mu(t)}{r^2 e^{-2Lt}} \Delta \eta(\cdot).
  \end{aligned}
\end{align}

  A simple calculation yields $\Delta (\psi^m) = m(m-1) \psi^{m-2} \abs{\nabla \psi}^2 + m
  \psi^{m-1} \Delta \psi$ and thus we can find $m_0 > 0$ independent of $(x,t)$ such that for $m \geq m_0$, $t \leq T$,
  \begin{align*}
    \abs{\Delta (\psi^m)} \leq  \frac{C(\eta) m^2}{r^2} e^{2Lt}(1- \delta)^{m-3} \psi < \e \psi.
  \end{align*}
  We have used $\abs{x - X(t, x_0)} \leq r e^{-Lt}$ and the regularity of $\eta$ to find $C(\eta)$.

  By Lipschitz continuity, we can estimate
  \begin{align}
    \label{transport_level}
    \begin{aligned}
    -\nabla \psi(x, t) \cdot \vec b(x) &\geq -\nabla \psi(x,t) \cdot \vec b(X(t, x_0)) - L
    \abs{\nabla \psi(x,t)} \abs{x - X(t, x_0)} \\
    &= \psi_t(x,t) - \mu'(t) \eta(\cdot),
    \end{aligned}
  \end{align}
  where the equality follows from \eqref{psi-derivatives} and the fact that $\eta(x)$ is nonincreasing in $\abs{x}$ and so
  $\nabla \psi \cdot (x - X(t, x_0)) = - \abs{\nabla \psi} \abs{x - X(t, x_0)}$.
  Putting everything together and using $\abs{F(x) - F(X(t, x_0))} < \e$, we have
  \begin{align}
    \label{subsol-check}
    \begin{aligned}
    \psi_t &\leq -\nabla \psi \cdot \vec b + \mu' \eta(\cdot)
    \leq -\nabla \psi \cdot \vec b
    + (F(X(t, x_0)) - 2\e) \psi\\
    &\leq -\nabla \psi \cdot \vec b + F \psi  + \Delta (\psi^m)
    = -\divo (\psi \vec b) + f\psi + \Delta(\psi^m), \qquad m \geq m_0.
    \end{aligned}
  \end{align}
  This concludes the subsolution part.

  The supersolution part follows from the same consideration, but using an upper bound in
  \eqref{transport_level} and then using the fact that $\eta(x)$ is nondecreasing in $\abs{x}$.
  Then \eqref{subsol-check} is adjusted to obtain a lower bound.
\end{proof}

\begin{remark}
  If we take $L$ strictly bigger then the Lipschitz constant of $\vec b$ then we get strict
  subsolution/supersolutions in Lemma~\ref{le:flow_barrier} as can be easily seen in
  \eqref{transport_level} and \eqref{subsol-check}, since we may assume that $\mu(t) > 0$, $\eta(0) >
  0$.
\end{remark}

\begin{remark}
  \label{re:pos-subbarrier}
  The solutions of \eqref{pme} can be approximated monotonically by positive solutions of
  \eqref{pme}. Therefore we need to check that a subbarrier is a subsolution of \eqref{pme} only
  for positive values.
\end{remark}

\begin{lemma}[Pressure barriers]
  \label{le:pressure_flow_barrier}
  Let $x_0 \in \Rn$, $T > 0$, $r > 0$. Set $N := \set{(x, t): \abs{x - X(t,
  x_0)} \leq r,\ t \in [0, T]}$. Let $L$ be the Lipschitz constant of $\vec b$ on $N$ and let
  $\kappa := \inf_N F / 2 > 0$.
  Suppose that $\mu = \mu_m(t) > 0$ satisfies $\mu' \leq \kappa (m-1) \mu$ and that $
  \pth{\frac{2n}{r^2} e^{2LT}} \max \mu \leq \kappa$.
  Then
  \begin{align}
    \label{pressure-streamline-barrier}
  \pi(x,t) = \mu(t) \pth{1 - \frac{|x - X(t, x_0)|^2}{r^2e^{-2Lt}}}
  \end{align}
  is a classical subsolution of \eqref{pmepressure} on  $\set{\pi > 0} \cap \Rn \times [0,T]$ for
  all $m > 1$.
\end{lemma}

\begin{proof}
  Let us write $\eta(x) := 1 - |x|^2$ for convenience.
  When $\pi(x,t) > 0$, the spatial derivatives are
  \begin{align*}
    \nabla \pi = \frac{1}{r e^{-Lt}} \mu \nabla \eta(\cdot), \qquad
    \Delta \pi = -\frac{2n \mu}{r^2 e^{-2Lt}}.
  \end{align*}
  Therefore by the assumption on $\mu$ we have
  \begin{align*}
    \Delta \pi + F \geq \kappa, \qquad m > 1.
  \end{align*}

  On the other hand, the time derivative can be expressed as
  \begin{align*}
    \pi_t(x, t) &= \mu'(t) \eta(\cdot) +\frac{1}{re^{-Lt}} \mu(t) \nabla \eta(\cdot) \cdot \pth{L (x - X(t, x_0)) -
    \vec b(X(t,x_0))}\\
    &= \mu'(t) \eta(\cdot) - L \abs{\nabla \pi(x, t)} \abs{x - X(t, x_0)} - \nabla \pi(x,t) \cdot \vec b(X(t, x_0)).
  \end{align*}
  Then the Lipschitz continuity of $\vec b$ and the assumption on $\mu$ yields
  \begin{align*}
    \pi_t(x,t) \leq \mu'(t)\eta(x) -\nabla \pi(x,t) \cdot \vec b(x)
    \leq (m - 1) \pi (\Delta \pi + F) -\nabla \pi \cdot \vec b + \abs{\nabla \pi}^2.
  \end{align*}
  Therefore $\pi$ is a strict classical subsolution of \eqref{pmepressure}.
\end{proof}

\subsection{Barriers up to \texorpdfstring{$\rho = 1$}{rho = 1}}
\label{sec:barriers-rho-one}

In this section we construct barriers for \eqref{pme} and \eqref{hs} valid up to density $\rho = 1$.
It seems rather difficult to construct explicit barriers, and we therefore rely on a convergence
result for radial solutions of the porous medium equation with a source and with no drift.
This construction is based on the results of \cite[Section~3]{KP}. There we proved that for a given
\emph{local monotone radial classical solution} $(\phi, \rho_\phi^E)$, defined below,
 of
\begin{align}
  \label{hsKP}
\left\{\begin{aligned}
    -\Delta \phi &= G(\phi) &&\text{ in } \{\phi>0\},\\
    V &= \dfrac{|\nabla\phi|}{(1-\rho_\phi^E)_+} &&\text{ on } \partial\{\phi>0\},\\
    \frac{\partial \rho_\phi^E}{\partial t} &= G(0) \rho_\phi^E,
\end{aligned}\right.
\end{align}
there exists a sequence of
of solutions $\rho_m$, $p_m = P_m(\rho_m)$ of the equation
\begin{equation}\label{pmeKP}
  \rho_t - \Delta (\rho^m) = \rho G(p), \quad \hbox{ in } \R^n\times \R^+
\end{equation}
such that $\rho_m$ converge in the sense of half-relaxed limits and $p_m$ converge uniformly
to the functions $\chi_{\set{\phi > 0}} + \rho_\phi^E \chi_{\set{\phi = 0}}$ and $\phi$, respectively, as $m
\to \infty$.
Here $G \in C^1(\R)$, $G(0) > 0$, $G'(s) < 0$ and such that there exists $p_M > 0$ with $G(p_M) =
0$.

More precisely, we say that a pair of functions $(\phi, \rho_\phi^E)$ is a \emph{local monotone
radial classical solution of \eqref{hsKP} on a cylindrical domain}
$Q := B_R(0) \times (0, T)$ or $Q := (\Rn \setminus B_R(0)) \times (0, T)$ for some $R > 0$, $T > 0$
if
\begin{enumerate}
  \item $\phi, \rho_\phi^E \in C_c(\cl Q)$, $\phi, \rho_\phi^E \geq 0$,
  \item $\phi \in C^2(\cl{\set{\phi > 0}})$, $\rho_\phi^E \in C^2(\cl Q)$,
  \item $\phi, \rho_\phi^E$ are radially symmetric in space with respect to the
origin,
  \item $\phi > 0$ on $\partial B_R(0) \times [0, T]$,
  \item $\phi$ and $\rho_\phi^E$ are nondecreasing in time,
  \item $\rho_\phi^E < 1$ on $\set{\phi = 0}$,
  \item if $\set{x = 0} \cap Q \neq \emptyset$, there exists a neighborhood of $x = 0$ on which
    $\phi(\cdot, t) = 0$ for all $t \in [0,T]$, and
  \item $\phi$, $\rho_\phi^E$ satisfy \eqref{hsKP} in $Q$ in the classical sense.
\end{enumerate}

The following theorem was proved in \cite{KP}.

\begin{theorem}[{cf. \cite[Theorem~3.4]{KP}}]
  \label{th:radial_approx}
  For a local monotone radial classical solution $(\phi, \rho_\phi^E)$ on $Q := B_R(0) \times (0,
  T)$ or $Q := (\Rn \setminus B_R(0)) \times (0, T)$ for some $R > 0$, $T > 0$ there exist
  nondecreasing-in-time radial solutions $\rho_m$, $p_m = P_m(\rho_m)$ of \eqref{pmeKP} such that $p_m
  \to \phi$ uniformly on $\cl Q$ and $\rho_m \to \chi_{\set{\phi > 0}} + \rho_\phi^E
  \chi_{\set{\phi = 0}}$ in the sense of the half-relaxed limits. We say that that $f_k \to
    f$ in the sense of the half-relaxed limits if $\halflimsup f_k = f^*$ and $\halfliminf f_k =
  f_*$.
\end{theorem}

\begin{proof}
  This statement was proved in \cite[Theorem~3.4]{KP} with the restriction that $Q$ does not
  contain a neighborhood of $\set{(x, t): |x| = 0}$. However, it is possible to extend
  that proof to handle the full cylinder $Q = B_R(0) \times(0,T)$. Indeed, we can take $\rho_m, p_m$
  positive and classical solutions of \eqref{pmeKP} with appropriately chosen initial and boundary
  data as explained in \cite{KP}. The only part of the proof
  that needs careful consideration is the uniform Lipschitz estimate for $p_m(\cdot, t)$ in a
  neighborhood of $\set{|x| = 0}$ since we work in polar coordinates in \cite{KP}. We therefore use
  the uniform Lipschitz estimate only on compact annuli $A \subset B_R(0) \setminus \set0$, which
  allows us to deduce the local uniform convergence $p_m \to \phi$ on $\cl Q \setminus \set{|x| =
  0}$. Then, as $\phi = 0$ near $|x| = 0$, a
  routine comparison of $p_m$ with a radial superbarrier of the type $c_1 t + c_2 - c_3 |x|^2$ at the origin
  yields the uniform convergence of $p_m$ on $\cl Q$.

  The proof of locally uniform convergence of $\rho_m$ to $\rho_\phi^E$ away from
  $\partial\set{\phi > 0}$ works near $|x| = 0$ with straightforward modification.  Finally, this
  argument can be extended to yield for any $\e > 0$ the existence of a neighborhood $\mathcal N$
  of $\partial\set{\phi > 0}$ such that $\rho_m > \rho_\phi^E - \e$ for all $m \gg 1$.
  We have $\halflimsup \rho_m \leq 1$ since $p_m$ is uniformly bounded.
  This implies the convergence in the sense of the half-relaxed limits.
\end{proof}

Let us explain how we can build barriers for \eqref{pme} from solutions of \eqref{pmeKP} for fixed
$m > 1$.
The main idea is to transport the solution of \eqref{pmeKP} along a streamline of $\vec b$, while
also using inf/sup-convolutions in space to account for the space variations of the drift field
$\vec b$.
Recall that $F := - \divo \vec b + f$ and $L$ is the Lipschitz constant of $\vec b$.
We will assume
that $(\rho, p)$ is a solution of \eqref{pmeKP} on the set $B_R(0) \times (-\tau, \tau)$ for
some $0 < r < R$, $\tau \in (0, \frac r{2LR})$
with a smooth source $G$ such that $G(\sup p) > \sup F$. For velocity perturbation $\alpha \in (L R, \frac r{2\tau})$, fixed point $z \in \Rn$, define
$r(t) := \frac r2 - \alpha t$ and the inf-convolution
\begin{align}
  \label{inf-superbarrier}
  w(x,t) := \inf_{|h| \leq r(t)} p(x - X(t, z) + h, t) \qquad \text{on $\cl Q$ where } Q := \set{(x,t): |x - X(t,
  z)| < R-r,\ |t| < \tfrac r{2\alpha}}.
\end{align}
$w$ is the inf-convolution of $p$ over a shrinking ball, following a characteristic of
\eqref{pme} going through the point $(z, 0)$.

Let us show that $w$ is a superbarrier for the pressure solution of \eqref{pme}.
Let $\phi$ be a continuous pressure solution of \eqref{pme} in $Q$ such that $\phi < w$ on
$\partial_P Q$.

Let us suppose that $\phi$ crosses $w$ from below at a point $(x_0, t_0) \in Q$, that is,
$\phi(x_0, t_0) = w(x_0, t_0)$ and $\phi \leq w$ for $t \leq t_0$. We show that this leads to a
contradiction.

We can assume that $\phi(x_0, t_0) = w(x_0, t_0) > 0$ since we can always approximate $p$ from above uniformly by
positive solutions of \eqref{pmeKP}.
By the definition of $w$, there exists $|h_0| \leq r(t_0)$ with $p(y_0, t_0) =
w(x_0, t_0)$ for $y_0 := x_0 - X(t_0, z) + h_0$. We also have $\nabla p(y_0, t_0) = \nabla \phi(x_0, t_0)$ and
$\Delta p(y_0, t_0) \geq \Delta \phi(x_0, t_0)$. For any $q
\in \Rn$, $|q| = 1$, we define
\begin{align*}
  h(t) := h_0 + q (r(t) - r(t_0)),
\end{align*}
which satisfies $|h(t)| \leq r(t)$ for all $t \in (-\tau, \tau)$.
In particular,
\begin{align}
  \label{line-order}
  \phi(x_0, t) \leq p(x_0 - X(t, z) + h(t), t), \qquad t \leq t_0,
\end{align}
with equality of $t = t_0$.
For $q = - \tfrac {\nabla p}{|\nabla p|} (y_0, t_0)$ if $\nabla p(y_0, t_0) \neq 0$ and $q = 0$
otherwise, the chain rule yields
\begin{align*}
  \phi_t(x_0, t_0) &\geq p_t(y_0, t_0)
  + \nabla p(y_0, t_0) \cdot (-\vec b(X(t_0, z)) + q r'(t_0))\\
  &\geq (m-1) p (\Delta p + G(p)) + |\nabla p|^2 + \nabla p \cdot (-\vec b(x_0) - q\alpha) - L|\nabla p| |x_0 - X(t_0, z)|\\
  &> (m-1) \phi (\Delta \phi + F(x_0)) + \nabla \phi \cdot (\nabla \phi - \vec b(x_0)) + |\nabla p| (\alpha - L|x_0 - X(t_0,
  z)|)\\
  &= \phi_t(x_0, t_0) + |\nabla p| (\alpha - L|x_0 - X(t_0, z)|),
\end{align*}
where we used $G(\sup p) > \sup F$.
Since the last term is nonnegative by the choice of $\alpha$ and $Q$, we arrive at a contradiction.
We conclude that $\phi$ cannot cross $w$ from below on $Q$ if the boundary data are ordered and
therefore it is a superbarrier for \eqref{pme}.

The construction of subbarriers follows the same idea, but we choose the source $G$ to satisfy
$0 < G(0) < \inf F$ and define $w$ as the sup-convolution by replacing $\inf$ by $\sup$
in \eqref{inf-superbarrier}. Other parameters are chosen as in the case of the superbarrier. Then
we suppose that $\phi > w$ on $\partial_P Q$. If there is a point $(x_0, t_0)$ with $\phi(x_0, t_0)
= w(x_0, t_0) > 0$, we again arrive at a contradiction. Note that now $\Delta p(y_0, t_0) \leq
\Delta \phi(x_0, t_0)$.
Finding $h_0$ and $h(t)$ as above, we get
the opposite inequality in \eqref{line-order}, which implies for $q = \tfrac {\nabla p}{|\nabla p|}
(y_0, t_0)$ or $q = 0$ as above
\begin{align*}
  \phi_t(x_0, t_0) &\leq p_t(y_0, t_0)
  + \nabla p(y_0, t_0) \cdot (-\vec b(X(t_0, z)) + q r'(t_0))\\
  &\leq (m-1) p (\Delta p + G(p)) + |\nabla p|^2 + \nabla p \cdot (-\vec b(x_0) - q\alpha) + L|\nabla p| |x_0 - X(t_0, z)|\\
  &< (m-1) \phi (\Delta \phi + F(x_0)) + \nabla \phi \cdot (\nabla \phi - \vec b(x_0)) - |\nabla p| (\alpha - |x_0 - X(t_0,
  z)|)\\
  &= \phi_t(x_0, t_0) - |\nabla p| (\alpha - |x_0 - X(t_0, z)|).
\end{align*}
The last term is nonpositive, and we again arrive at a contradiction. Therefore $w$ is a subbarrier
for \eqref{pme}.

\medskip

Note that the above construction is independent of $m > 1$. Furthermore, if $p_m \to p_\infty$
uniformly and $\rho_m \to \rho_\infty$ in the sense of half-relaxed limits,
so do the respective inf- and sup-convolutions.
We use these facts to construct sequence of nice barriers. First, we recall that we can always
construct simple solutions of \eqref{hsKP}.

\begin{lemma}
  \label{le:existence-limit-radial}
  For any positive constants $\eta,r_0$ and $\rho^0 \in [0, 1)$ and a function $G \in C^\infty(\R)$ there exist $\delta > 0$ and a
  local monotone radial classical solution $(\phi, \rho_\phi^E)$ of \eqref{hsKP} on the (interior)
  cylindrical domain $Q:=B_{r_0
  + \delta}(0) \times (0, 2 \delta)$ such that $|\nabla \phi| = \eta$ on $\partial \set{\phi > 0}$,
  $\rho_\phi^E(\cdot, \delta) = \rho$, $\set{\phi(\cdot, \delta) = 0} = \cl B_{r_0}(0)$.

  Similarly, such a solution exists on the (exterior) cylindrical domain $Q = B_{r_0 - \delta}(0) \times (0,
  2\delta)$ with $\set{\phi(\cdot, \delta) = 0} = \Rn \setminus B_{r_0}(0)$.
\end{lemma}

\begin{proof}
  Let us only construct the solution on the interior cylinder, the exterior is analogous.
  The solutions can be constructed using the ODE theory. We find the solution $r = r(t)$
  of the ODE $r'(t) = -\frac \eta {1 - \rho e^{G(0)t}}$ with $r(0) = r_0$ which exists, is smooth and is
  positive, in a neighborhood of $t = 0$. Then for every time $t$ we define $u = u_t(s)$ to be the
  solution of the ODE given by writing $-\Delta \phi = G(\phi)$ in radial coordinates, with
  initial condition
  $u(r(t)) = 0$ and $u'(r(t)) = \eta$. This solution exists and is positive on $s > r(t)$ on a
  neighborhood of $s = r(t)$. It is also smooth, and depends smoothly on $t$.
  We then take $\delta > 0$ so that the functions $\phi(x, t) := u_{t - \delta}(|x|)$ when $|x|
  \geq r(t - \delta)$ and zero otherwise, and $\rho_\phi^E = \rho e^{G(0) (t-\delta)}$ satisfy the
  assumptions on a local monotone radial classical solution.
\end{proof}

Using the convergence result for radial solutions, Theorem~\ref{th:radial_approx}, and the above
construction, given any classical barrier of \eqref{hs} in the sense of
Definition~\ref{def:barrier} and a point on its free boundary, we can
create a sequence of nice solutions of the $m$-problems that converge to a function that touches
the barrier at the given boundary point.

\begin{proposition}
  \label{pr:barrier-sequence}
  Let $U$ be an open set, $(x_0, t_0) \in U$ and let $\phi \in C^{2,1}(U)$, $\rho^E\in C(U)$ satisfy $\phi(x_0,t_0) =
  0$, $|\nabla \phi|(x_0,t_0) \neq 0$, $\rho^E(x_0,t_0) < 1$ and  $\phi_t > \frac{\abs{\nabla\phi}^2}{(1-\rho^E)_+} -
  \vecb \cdot \nabla\phi$ at $(x_0,t_0)$.
  Then there exists a parabolic neighborhood $\mathcal N$ of
  $(x_0, t_0)$ and a sequences $\varphi_\rho^m$, $\varphi_p^m =
  P_m(\varphi_\rho^m)$, of classical supersolutions of \eqref{pme} and functions $\varphi_p$,
  $\varphi_\rho$ such that $\cl{\set{\varphi_p > 0}} = \set{\varphi_\rho = 1}$, $\varphi_p^m \to \varphi_p$ uniformly on $\mathcal N$,
  $\varphi_\rho^m \to \varphi_\rho$ in the sense of half-relaxed limits, and
  $\varphi_p \geq \phi$ and $\varphi_\rho \geq \chi_{\cl{\set{\phi > 0}}} + \rho^E
  \chi_{\cl{\set{\phi > 0}}^\compl}$ on $\mathcal N$, and $\varphi_p(x_0, t_0) = 0$.

  An analogous sequence exists for a subbarrier, i.e., if $\phi_t > \frac{\abs{\nabla\phi}^2}{(1-\rho^E)_+} -
  \vecb \cdot \nabla\phi$ at $(x_0,t_0)$, and the limit then satisfies $\varphi_p \leq \phi$, $\varphi_\rho \leq
  \chi_{\cl{\set{\phi > 0}}} + \rho^E \chi_{\cl{\set{\phi > 0}}^\compl}$.
\end{proposition}

\begin{proof}
  Let us again show this only for the superbarrier, subbarrier is analogous.
  We shall construct the limit functions $\varphi_p$ and $\varphi_\rho$ first using
  Lemma~\ref{le:existence-limit-radial} and \eqref{inf-superbarrier}.

  By translating everything, we can for simplicity assume that $(x_0, t_0) = (0, 0)$.
  Since $\phi$ is $C^2$ in space and $\nabla \phi(0, 0) \neq 0$, $\set{\phi(\cdot, 0) > 0}$ has an
  exterior ball property at $0$.
  Let $\nu$ be the outer unit normal of $\set{\phi(\cdot, 0) > 0}$ at $x = 0$.

  Recall that $L$ is the Lipschitz constant of $\vecb$.
  For every $\e > 0$, let us set the following parameters: $z = \e \nu$, $r_0 = \frac {20}{19}\e$,
  $\tilde r_0 = \frac{r_0}{10}$, $\alpha = L r + \e$, $\eta = \abs{\nabla \phi}(0,0) + \e$ and
  $\rho^0 = \sup_{B_{2r}(z)} \rho^E(\cdot, 0) + \e$.
  We chose the parameters so that $r_0 - \frac{\tilde r_0}2 = \e$.
  Let us also take $G(s) = \sup F + 1 - s$.

  According to Lemma~\ref{le:existence-limit-radial}, there is a local monotone radial solution
  $(\zeta, \rho_\zeta^E)$ of \eqref{hsKP} on the set $B_{r_0 + \delta}(0) \times (-\delta, \delta)$
  for some $\delta > 0$ (depending on $\e > 0$). We can take $\rho_\zeta^E$ that does not depend on $x$.
  Let us define $w$ as in \eqref{inf-superbarrier} with $p = \zeta$.
  By the choice of parameters, $\set{w(\cdot, 0) = 0}$ is an exterior ball of $\set{\phi(\cdot, 0)
  > 0}$ at $(0, 0)$.
  Moreover, by construction, the normal velocity of $\set{w > 0}$ at $(0,0)$ is $V_w = \frac{\eta}{1 - \rho^0} +
  \alpha + \vecb(z, 0) \cdot \nu$. On the other hand, the normal velocity of $\set{\phi > 0}$ at
  $(0,0)$ satisfies $V_\phi > \beta := \frac{\abs{\nabla \phi(0,0)}}{1 - \rho^E(0,0)} + \vecb(0, 0) \cdot \nu$.
  By continuity, $V_w$ converges to $\beta$ as $\e \to 0+$ and therefore $V_w < V_\phi$ for sufficiently small $\e > 0$. Since
also $\abs{Dw}(0,0) = \eta > \abs{D\phi(0,0)}$, we
  conclude that $\phi - w$ has a strict maximum $0$ in the set $\cl{\set{\phi > 0}} \cap \set{t
  \leq 0}$ at $(0,0)$ for
  $\e > 0$ sufficiently small.

  In particular, there exists a parabolic neighborhood $\mathcal N$ of $(0,0)$ on which we have
  $\varphi_p := w \geq \phi$ and $\varphi_\rho := \chi_{\cl{\set{w > 0}}} + \rho^E_\zeta
    \chi_{\cl{\set{w > 0}}^\compl} \geq \chi_{\cl{\set{\phi > 0}}} + \rho^E
    \chi_{\cl{\set{\phi > 0}}^\compl}$.

    Finally, let $\rho_m$ and $p_m$ be the solutions of \eqref{pmeKP} provided by
    Theorem~\ref{th:radial_approx} for $(\zeta, \rho^E_\zeta)$ above. Let $\varphi_\rho^m$ and
    $\varphi_p^m$ be their inf-convolutions as in \eqref{inf-superbarrier}. Then, by making
    $\mathcal N$ smaller if necessary (independent of $m$), we have that $\varphi_\rho^m$ and
    $\varphi_p^m$ are classical super solutions of \eqref{pme} on $\mathcal N$ that converge to
    $\varphi_\rho$, $\varphi_p$ as required.
\end{proof}

To be able to use the sequence of barriers, we state the following technical lemma.

\begin{lemma}
  \label{le:barrier-order}
  Suppose that $\rho_m, p_m = P_m(\rho_m)$ are USC and $v_m, u_m = P_m(v_m)$ are LSC nonnegative,
  uniformly bounded functions. Set
  $p = \halflimsup_{m\to\infty} p_m$, $\rho = \halflimsup_{m\to\infty} \rho_m$, $u =
  \halfliminf_{m\to\infty} u_m$, $v = \halfliminf_{m\to\infty} v_m$.
  Suppose that $K$ is a compact set.
  If $p < u$ in $\set{u > 0} \cap K$ and $\rho < v$ in $\set{\rho < 1} \cap K$, and $K \subset \set{u > 0} \cup
  \set{\rho < 1}$, then $p_m \leq u_m$ and $\rho_m \leq v_m$ for large $m$.
\end{lemma}

\begin{proof}
  Since $K$ is compact, it is enough to show that if $\xi_k \to \xi$, $\xi_k \in K$, $\xi \in K$,
  and $m_k \to \infty$, then $p_{m_k}(\xi_k) \leq u_{m_k}(\xi_k)$ for large $k$.

  If $u(\xi) > 0$, then $(u - p)(\xi) > 0$ and therefore by the half-relaxed convergence
  $(u_{m_k} - p_{m_k})(\xi_k) > 0$ for $k$ large enough.
  Similarly, if $\rho(\xi) < 1$, then $(v - \rho)(\xi) > 0$ and therefore by half-relaxed convergence
  $(v_{m_k} - \rho_{m_k})(\xi_k) > 0$ for $k$ large enough.
\end{proof}

\subsection{Uniform bounds}

\begin{lemma}
  \label{le:global-upper-bound}
  Suppose that $\rho^0$ and $F = f - \divo \vecb$ are bounded on $\Rn$. Then $\rho_m \leq R
  e^{M t}$ on $\Rn \times [0, \infty)$, where $R = \sup \rho^0$, $M = \sup F$.
\end{lemma}

\begin{proof}
  Note that $\psi(x,t) := R e^{M t}$ is a supersolution of \eqref{pme}.
\end{proof}

\begin{lemma}
  \label{le:pressure_uniform_bound}
  Suppose that the initial data $\rho^0$ is compactly supported. Then $p_m$ are bounded uniformly in
  $m$, locally in time, and $\rho = \limsup^* \rho_m \leq 1$.
\end{lemma}

\begin{proof}
Take a superbarrier $\Pi(x,t) := (R^2(t) - K\abs{x}^2)_+$ for
large enough $K$ so that $2n K > \sup |F|$ and sufficiently fast growing $R$.
Bound on $\rho$ follows from $\rho_m = P_m^{-1}(p_m)$.

\end{proof}

\subsection*{Acknowledgments}

I. K. is partially supported by NSF grant DMS-1566578. N.P. is partially supported by Japan
Society for the Promotion of Science (JSPS) KAKENHI Grant No. 26800068 (Wakate B).

\end{document}